\documentclass[12pt,oneside]{article}

\usepackage{amsmath}
\usepackage{amssymb}
\usepackage{color}
\usepackage{rotating}

\usepackage{graphicx}
\usepackage{subfigure}

\usepackage{url}

\usepackage{appendix}

\author{
Jing Liu\footnote{E-mail address: 2022201091@buct.edu.cn}\\ 
\footnotesize{College of Mathematics and Physics, }\\ 
\footnotesize{Beijing University of Chemical Technology, Beijing 100029, China}
\vspace{1em} \\ 
Guang Rao\footnote{E-mail address: grant.r@nus.edu.sg}\\ 
\footnotesize{Yale-NUS College, National University of Singapore, }\\ 
\footnotesize{16 College Avenue West, Singapore 138527, Singapore} 
\vspace{1em} \\ 
Hui Zhou\footnote{E-mail addresses: zhouh06@qq.com, zhouhlzu06@126.com}~\footnote{Corresponding author.}\\ 
\footnotesize{College of Mathematics and Physics, }\\ 
\footnotesize{Beijing University of Chemical Technology, Beijing 100029, China}
}

\title{\Large Oriented diameter of the complete tripartite graph (III)}

\def\d{{\rm d}}
\def\diam{{\rm diam}}
\def\K{{\sf K}}

\def\rev{{\sf rev}}

\newtheorem{theorem}{Theorem}[section]

\newtheorem{lemma}[theorem]{Lemma}

\newtheorem{orientation}[theorem]{Orientation}

\newcommand*{\QEDA}{\hfill\ensuremath{\blacksquare}}  
\newenvironment{proof}[1][\hspace{2ex}\textbf{\textit{Proof}.}\hspace{1ex}]{\begin{trivlist}\item[\hskip \labelsep {\bfseries #1}]}{\QEDA\end{trivlist}}

\begin{document}

\maketitle

\begin{abstract}
Given a graph $G$, let $\mathbb{D}(G)$ be the set of all strong orientations of $G$, and define the oriented diameter of $G$ to be
\begin{center}
$f(G)=\min\{\diam(D)\mid D\in \mathbb{D}(G)\}$.
\end{center}
In this paper, we determine the oriented diameter of complete tripartite graph $\K(3,p,q)$ for $p\geqslant 5$. Combining with the previous results, the oriented diameter of complete tripartite graph $\K(3,p,q)$ are known.
\end{abstract}




\section{Introduction}\label{sec-Introduction}

Let $G$ be a finite undirected connected graph with vertex set $V(G)$ and edge set $E(G)$. Take $u,v\in V(G)$. The distance $\d_G(u,v)$ is the number of edges in a shortest path connecting $u$ and $v$ in $G$. The diameter of $G$ is defined to be $\diam(G)=\max\{\d_G(u,v)\mid u,v\in V(G)\}$. An edge $e\in E(G)$ is called a bridge if the resulting graph obtained from $G$ by deleting $e$ is disconneccted. A graph is called bridgeless if it has no bridge. An orientation $D$ of $G$ is a digraph obtained from $G$ by assigning a direction for each edge. A digraph is strong (or strongly connected) if for any two vertices $u,v$, there is a directed path from $u$ to $v$ in this digraph. An orientation $D$ of $G$ is called a strong orientation if the digraph $D$ is strong. Robbins' one-way street theorem~\cite{Robbins1939} states that
\begin{center}a connected graph has a strong orientation if and only if it is bridgeless.\end{center}

Given a connected graph $G$ which is bridgeless, let $\mathbb{D}(G)$ be the set of all strong orientations of $G$. Define the oriented diameter of $G$ to be
\begin{center}
$f(G)=\min\{\diam(D)\mid D\in \mathbb{D}(G)\}$,
\end{center}
where $\diam(D)$ denote the diameter of $D$. The problem of evaluating oriented diameter $f(G)$ of an arbitrary connected graph $G$ is very difficult. Chv\'{a}tal and Thomassen~\cite{ChvatalThomassen1978} showed that the problem of deciding whether a graph admits an orientation of diameter two is NP-hard.

Given positive integers $n$, $p_1,p_2,\ldots,p_n$, let $\K_n$ denote the complete graph of order $n$, and let $\K(p_1,p_2,\ldots,p_n)$ denote the complete $n$-partite graph having $p_i$ vertices in the $i$-th partite set $V_i$ for each $i\in \{1,2,\ldots,n\}$. Thus $\K_n$ is also a complete $n$-partite graph $\K(p_1,p_2,\ldots,p_n)$ where $p_1=p_2=\cdots=p_n=1$. The oriented diameter of complete graph $\K_n$ was obtained by Boesch and Tindell~\cite{BoeschTindell1980}: 
\[
f(\K_n)=\left\{\begin{array}{ll}
                        2, & \text{if $n\geqslant 3$ and $n\neq 4$};\\ 
                        ~ & ~\\
                        3, & \text{if $n=4$}.
                   \end{array}
             \right.
\]
The oriented diameter of complete bipartite graph $\K(p,q)$ for $2\leqslant p\leqslant q$ was obtained by \v{S}olt\'{e}s~\cite{Soltes1986}: 
\[
f(\K(p,q))=\left\{\begin{array}{ll}
                        3, & \text{if $q\leqslant \binom{p}{\lfloor\frac{p}{2}\rfloor}$};\\ 
                        ~ & ~\\
                        4, & \text{if $q> \binom{p}{\lfloor\frac{p}{2}\rfloor}$};
                   \end{array}
             \right.
\]
where $\lfloor x\rfloor$ denotes the greatest integer not exceeding $x$. Gutin~\cite{Gutin1989} also gave a beautiful proof of this result, independently, by using Sperner's Lemma~\cite{Sperner1928,vanLintWilson1992}. Two sets $A$ and $B$ are independent if $A\setminus B\neq\emptyset$ and $B\setminus A\neq\emptyset$. If $A$ and $B$ are independent, then we say that $A$ is independent with $B$ or $B$ is independent with $A$. Denote $\lceil x \rceil$ the least integer exceeding $x$.

\begin{lemma}[Sperner's Lemma]\label{lem-Sperner}
Let $n$ be a positive integer, and let $\mathbb{C}$ be a collection of subsets of $[n]=\{1,2,\ldots,n\}$ such that $A$ and $B$ are independent for any distinct sets $A,B\in \mathbb{C}$. We call $\mathbb{C}$ an independent collection of $[n]$. Then $|\mathbb{C}|\leqslant \binom{n}{\lfloor\frac{n}{2}\rfloor}$, and the equality holds if and only if all subsets in $\mathbb{C}$ are of the same size, $\lfloor\frac{n}{2}\rfloor$ or $\lceil\frac{n}{2}\rceil$. 
\end{lemma}

For $n\geqslant 3$, Plesn\'{\i}k~\cite{Plensik1985}, Gutin~\cite{Gutin1994}, and Koh and Tan~\cite{KohTan1996DM} obtained independently the following result for oriented diameter of complete $n$-partite graph:  
\[
2\leqslant f(\K(p_1,p_2,\ldots,p_n))\leqslant 3.
\]
They also got some other results on complete multipartite graphs. In a survey by Koh and Tay~\cite{KohTay2002}, earlier results were collected in it, for example: lower and upper bounds for some graphs with special parameters, oriented diameters for Cartesian product and extensions of graphs, etc.  

\vspace{2em}

Koh and Tay~\cite{KohTay2002} proposed a problem:
\begin{center}given a complete multipartite graph $G=\K(p_1,p_2,\ldots,p_n)$,\\ classify it according to $f(G)=2$ or $3$.\end{center} 

\vspace{2em}

For complete tripartite graphs $\K(2,p,q)$ with $2\leqslant p\leqslant q$, this problem is completely solved. Koh and Tan~\cite{KohTan1996GC} proved 
\begin{center}$f(\K(2,p,q))=2$ for $2\leqslant p\leqslant q\leqslant \binom{p}{\lfloor\frac{p}{2}\rfloor}$. \end{center}
Rajasekaran and Sampathkumar~\cite{RajasekaranSampathkumar2015} proved $f(\K(2,2,q))=3$ for $q\geqslant 3$, and $f(\K(2,3,q))=3$ for $q\geqslant 4$, and they also got in an unpublished manuscript that $f(\K(2,4,q))=3$ for $q\geqslant 7$. Hence they~\cite{RajasekaranSampathkumar2015} conjectured that  
\begin{center}
$f(\K(2,p,q))=3$ when $p\geqslant 5$ and $q>\binom{p}{\lfloor\frac{p}{2}\rfloor}$. 
\end{center}
Liu, Rao, Zhang and Zhou~\cite{LiuRaoZhangZhou2024+} proved
\begin{center}
$f(\K(2,p,q))=3$ when $p\geqslant 2$ and $q>\binom{p}{\lfloor\frac{p}{2}\rfloor}$,
\end{center}
and hence they confirmed this conjecture. Combining with these results, the oriented diameter of complete tripartite graph $\K(2,p,q)$ is completely determined: for $2\leqslant p\leqslant q$, 
\[
f(\K(2,p,q))=\left\{\begin{array}{ll}
                        2, & \text{if $p\leqslant q\leqslant \binom{p}{\lfloor\frac{p}{2}\rfloor}$};\\ 
                        ~ & ~\\
                        3, & \text{if $q> \binom{p}{\lfloor\frac{p}{2}\rfloor}$}.
                   \end{array}
             \right.
\]


The problem goes to complete tripartite graphs $\K(3,p,q)$ with $3\leqslant p\leqslant q$. Liu and Zhou~\cite{LiuZhou2024+} proved
\[
f(\K(3,3,q))=\left\{\begin{array}{ll}
                        2, & \text{if $3\leqslant q\leqslant 6$};\\ 
                        ~ & ~\\
                        3, & \text{if $q> 6$};
                   \end{array}
             \right.
\]
\[
f(\K(3,4,q))=\left\{\begin{array}{ll}
                        2, & \text{if $4\leqslant q\leqslant 11$};\\ 
                        ~ & ~\\
                        3, & \text{if $q> 11$}.
                   \end{array}
             \right.
\]
In this paper, we continue to solve this problem for $\K(3,p,q)$ with $p\geqslant 5$, and we get the following Theorem~\ref{thm-main}. Hence the problem posed by Koh and Tay for complete tripartite graphs $\K(3,p,q)$ is completely solved. 

\begin{theorem}\label{thm-main}
Suppose $5\leqslant p\leqslant q$, then 
\[
f(\K(3,p,q))=\left\{\begin{array}{ll}
                        2, & \text{if $p\leqslant q< \binom{p+1}{\lfloor\frac{p+1}{2}\rfloor}$};\\ 
                        ~ & ~\\
                        3, & \text{if $q\geqslant \binom{p+1}{\lfloor\frac{p+1}{2}\rfloor}$}.
                   \end{array}
             \right.
\]
\end{theorem}

Our main theorem Theorem~\ref{thm-main} has two parts.
\begin{enumerate}
  \item[(1)] {\bf big part:} when $q\geqslant \binom{p+1}{\lfloor\frac{p+1}{2}\rfloor}$, then the oriented diameter of $\K(3,p,q)$ is three, and the proof of this part is by analysing the structure of strong orientations with diameter two; 
  \item[(2)] {\bf small part:} when $p\leqslant q< \binom{p+1}{\lfloor\frac{p+1}{2}\rfloor}$, we find strong orientations with diameter two.
\end{enumerate}


\section{Preliminaries}\label{sec-priminaries}

Let $D$ be a digraph with vertex set $V(D)$. If $u,v\in V(D)$, the distance $\partial_D(u,v)$ is the number of directed edges in a shortest directed path from $u$ to $v$ in $D$. If $D$ is strongly connected, the diameter of $D$ is defined as 
\begin{center}
$\diam(D)=\max\{\partial_D(u,v)\mid u,v\in V(D)\}$, 
\end{center} 
and similarly, for subsets $U,V\subseteq V(D)$, we denote 
\begin{center}
$\diam(U,V)=\max\{\partial_D(u,v)\mid u\in U,~v\in V\}$, 
\end{center} 
when $U=\{u\}$ we write $\diam(u,V)$ for $\diam(U,V)$, when $V=\{v\}$ we write $\diam(U,v)$ for $\diam(U,V)$. 

Let $u,v\in V(D)$, and $U,V\subseteq V(D)$ such that $U\cap V=\emptyset$. We write 
\begin{center}
`$u\rightarrow v$' if the direction is from $u$ to $v$ in $D$. 
\end{center}
We write 
\begin{center}
`$U\rightarrow V$' if $x\rightarrow y$ for each $x\in U$ and for each $y\in V$, 
\end{center} 
if $U=\{u\}$ we write `$u\rightarrow V$' for $U\rightarrow V$, if $V=\{v\}$ we write `$U\rightarrow v$' for $U\rightarrow V$. All the out-neighbors of $u$ form a set 
\begin{center}
$N_D^+(u)=\{w\in V(D)\mid u\rightarrow w\}$, 
\end{center} 
and all the in-neighbors of $v$ form a set 
\begin{center}
$N_D^-(v)=\{w\in V(D)\mid w\rightarrow v\}$. 
\end{center}
For $S\subseteq V(D)$, we use $D[S]$ to denote the subgraph induced by $S$ in $D$, and denote the neighbour sets 
\begin{center}
  $N_D^+(S)=\bigcup\limits_{x\in S}N_D^+(x)$,\\ 
  $N_D^-(S)=\bigcup\limits_{x\in S}N_D^-(x)$. 
\end{center} 
We use $\rev(D)$ to denote the directed graph obtained from $D$ by reversing all the directions of arcs (directed edges) of $D$.

\begin{lemma}\label{lem-neighbour-diamter-2}
Suppose $X$ is a strongly connected digraph. Let $u,v\in V(X)$ be two vertices of $X$. If $N_X^+(u)\cap N_X^-(v)=\emptyset$, then $\partial_X(u,v)\neq 2$.
\end{lemma}

\begin{proof}
We assume $\partial_X(u,v)=2$, then there exists $w\in V(X)$ such that $u\rightarrow w\rightarrow v$. So $w\in N_X^+(u)\cap N_X^-(v)\neq\emptyset$, a contradiction.
\end{proof}

We can read out the following lemma from the proof of the oriented diameter of complete bipartite graph.

\begin{lemma}\label{lem-Kpq-distance-4}
Let $G=\K(p,q)$ be the complete bipartite graph with vertex set $V(G)=V_1\cup V_2$, where $2\leqslant |V_1|=p\leqslant |V_2|=q$. If $q>\binom{p}{\lfloor\frac{p}{2}\rfloor}$, then for any orientation $D$ of $G$, there exist vertices $z,w\in V_2$ such that $\partial_D(z,w)\geqslant 4$.
\end{lemma}

Let $P$ be a set, we use $2^P$ to denote the power set of $P$, i.e., $2^P$ consists of all the subsets of $P$. For a positive integer $k$, we use $\binom{P}{k}$ to denote the collection of subsets of $P$ with size $k$, i.e., $\binom{P}{k}=\{A\subseteq P\mid |A|=k\}$. In the following, we give a short list of the binomial coefficients $\binom{n}{\lfloor\frac{n}{2}\rfloor}$.

\begin{table}[!ht]
\centering
    \begin{tabular}{c|c|c|c|c|c|c|c|c|c|c|c|c|c}
        \hline 
        $n$ & 1 & 2 & 3 & 4 & 5 & 6 & 7 & 8 & 9 & 10 & 11 & 12 & $\cdots$\\
        \hline 
        $\binom{n}{\lfloor\frac{n}{2}\rfloor}$ & 1 & 2 & 3 & 6 & 10 & 20 & 35 & 70 & 126 & 252 & 462 & 924 & $\cdots$ \\
        \hline
    \end{tabular}
\end{table}


\section{Notations of $\K(3,p,q)$}\label{sec-notation-structure-K3pq}


In this section, we give some notations and structural analysis of diameter two orientations of $\K(3,p,q)$. Let 
\begin{center}
$V_1=\{x_1, x_2, x_3\}$, ~~~~~\\
$V_2=\{y_1, y_2, \ldots, y_p\}$,\\
$V_3=\{z_1, z_2, \ldots, z_q\}$~~
\end{center} 
be the three parts of the vertex set of $\K(3,p,q)$ with $3\leqslant p\leqslant q$. Let
\begin{center}
$[3]=\{1,2,3\}$.
\end{center}
Let $D$ be a strong orientation of $\K(3,p,q)$. Let  
\begin{center}
$N_D^{+}=N_D^+(x_1)\cap N_D^+(x_2)\cap N_D^+(x_3)$,\\
$N_D^{-}=N_D^-(x_1)\cap N_D^-(x_2)\cap N_D^-(x_3)$.
\end{center}
For $A\subseteq [3]$ with $\emptyset \neq A\neq [3]$, we let 
\begin{center}
$N_D^{A}=\Biggl(\bigcap\limits_{i\in A}N_D^+(x_i)\Biggr)\cap \Biggl(\bigcap\limits_{j\in [3]\setminus A}N_D^-(x_j)\Biggr)$.
\end{center}
We can also write $N_D^{[3]}=N_D^+$ and $N_D^{\emptyset}=N_D^-$. For $i\in \{2,3\}$, and $A\subseteq [3]$, let 
\begin{center}
$V_i^A=V_i\cap N_D^A$.
\end{center}
We also write $V_i^+=V_i^{[3]}$ and $V_i^-=V_i^{\emptyset}$. By definition, we have 
\begin{center}
$V_i^-\rightarrow V_1\rightarrow V_i^+$, 
\end{center} 
and 
\begin{center}
$\{x_i\mid i\in A\}\rightarrow V_i^A\rightarrow \{x_i\mid i\in [3]\setminus A\}$ for $A\subseteq [3]$ with $\emptyset \neq A\neq [3]$. 
\end{center} 
Hence $\{V_i^A\mid A\subseteq [3]\}$ form a partition of $V_i$. When $A=\{\alpha\}\subseteq [3]$, we write $V_i^\alpha$ for $V_i^A$; when $A=\{\alpha,\beta\}\subseteq [3]$, we write $V_i^{\alpha\beta}$ for $V_i^A$. So we can write
\begin{center}
$\{V_i^A\mid A\subseteq [3]\}=\{V_i^+, V_i^1, V_i^2, V_i^3, V_i^{12}, V_i^{13}, V_i^{23}, V_i^-\}$.
\end{center}


\begin{lemma}\label{lem-diam2-ViA-VjB}
Let $3\leqslant p\leqslant q$, and let $D$ be a strong orientation of $\K(3,p,q)$ with diameter two. Suppose $\{i,j\}=\{2,3\}$ and $A\subseteq [3]$.
\begin{enumerate}
  \item[(1)] If $V_i^A\neq \emptyset$, then $V_j^A=\emptyset$.
  \item[(2)] Suppose $V_i^A\neq \emptyset$ and for any $B\subseteq [3]$, we have either $V_i^A\rightarrow V_j^B$ or $V_j^B\rightarrow V_i^A$, then $|V_i^A|=1$.
  \item[(3)] If $V_i^{+}\neq\emptyset$, then $V_i^{+}\rightarrow V_j$ and $|V_i^{+}|=1$;\\ 
      if $V_i^{-}\neq\emptyset$, then $V_j\rightarrow V_i^{-}$ and $|V_i^{-}|=1$.
  \item[(4)] Let $T_1=\{B\subseteq [3]\mid V_i^A\rightarrow V_j^B\text{ or }V_j^B\rightarrow V_i^A\}$ and $T_2=2^{[3]}\setminus T_1$. The set $V_j$ is partitioned into two sets $V_{j1}=\bigcup\limits_{B\in T_1}V_j^B$ and $V_{j2}=\bigcup\limits_{C\in T_2}V_j^C$. Suppose $q_A=|V_i^A|\geqslant 1$ and $p_2=|V_{j2}|\geqslant 1$. Then $q_A\leqslant \binom{p_2}{\lfloor\frac{p_2}{2}\rfloor}$. Here we call $(V_{j2},V_i^A)$ a mixed pair, which would be used frequently later. 
  \item[(5)] Take $B\subseteq [3]$ such that $A\subseteq B$ and $A\neq B$. If $V_i^A\neq\emptyset$ and $V_j^B\neq\emptyset$, then $V_j^B\rightarrow V_i^A$. In particular, let $\alpha,\beta\in [3]$ with $\alpha\neq \beta$, if $V_i^\alpha\neq\emptyset$ and $V_j^{\alpha\beta}\neq\emptyset$, we have $V_j^{\alpha\beta}\rightarrow V_i^\alpha$; furhtermore, if $V_i^{+}\neq\emptyset$ and $V_j^{-}\neq\emptyset$, then $V_i^+\rightarrow V_j^{\alpha\beta}\rightarrow V_i^\alpha\rightarrow V_j^-$. 
\end{enumerate} 
\end{lemma}

\begin{proof}
Suppose $V_i^A\neq \emptyset$ and $V_j^A\neq\emptyset$. Take $u\in V_i^A$ and $v\in V_j^A$. We have either $u\rightarrow v$ or $v\rightarrow u$. The proof of these two cases are similar. We may assume $v\rightarrow u$, and so $\partial_D(u,v)\geqslant 2$. Then $N_D^+(u)\subseteq B\cup V_j$ and $N_D^-(v)\subseteq A\cup V_i$ where $B=[3]\setminus A$. Hence $N_D^+(u)\cap N_D^-(v)=\emptyset$. By Lemma~\ref{lem-neighbour-diamter-2}, we have $\partial_D(u,v)\neq 2$. Thus $\partial_D(u,v)\geqslant 3$, which is a contradiction. 

Let $S_1=\{B\subseteq [3]\mid V_i^A\rightarrow V_j^B\}$ and $S_2=\{B\subseteq [3]\mid V_j^B\rightarrow V_i^A\}$. Then $V_2$ is partitioned into two sets $V_{j1}=\bigcup\limits_{B\in S_1}V_j^B$ and $V_{j2}=\bigcup\limits_{B\in S_2}V_j^B$. We prove this statement by contradiction, and we suppose $|V_i^A|\geqslant 2$. Take $z_i,w_i\in V_i^A$, then $\partial_D(z_i,w_i)\geqslant 2$. Now $N_D^+(z_i)\subseteq V_1^+\cup V_{21}$ and $N_D^-(w_i)\subseteq V_1^-\cup V_{22}$ where $V_1^+=\{x_i\mid i\in [3]\setminus A\}$ and $V_1^-=\{x_j\mid j\in A\}$. Hence $N_D^+(z_i)\cap N_D^-(w_i)=\emptyset$. By Lemma~\ref{lem-neighbour-diamter-2}, we have $\partial_D(z_i,w_i)\neq 2$, and so $\partial_D(z_i,w_i)\geqslant 3$, which is a contradiction. 

Suppose $V_i^{+}\neq\emptyset$. Take any $y\in V_i^{+}$ and any $z\in V_j$. If $z\rightarrow y$, then $\partial_D(y,z)\geqslant 2$. We know $N_D^+(y)\subseteq V_j\setminus \{z\}$ and $N_D^-(z)\subseteq V_1\cup (V_i\setminus \{y\})$. So $N_D^+(y)\cap N_D^-(z)=\emptyset$. By Lemma~\ref{lem-neighbour-diamter-2}, we have $\partial_D(y,z)\neq 2$, and so $\partial_D(y,z)\geqslant 3$, a contradiction. Hence $y\rightarrow z$. This means $V_i^{+}\rightarrow V_j$. By the second statement of this lemma, we have $|V_i^{+}|=1$. The proof for the case $V_i^{-}\neq\emptyset$ is analogous. 

Let $F=D[V_{j2}\cup V_i^A]$, then $F$ is an orientation of $\K(p_2,q_A)$. Let $T_1^+=\{B\in T_1\mid V_i^A\rightarrow V_j^B\}$ and $T_1^-=\{B\in T_1\mid V_j^B\rightarrow V_i^A\}$. Then $V_{j1}$ is partitioned into two sets $V_{j1}^+=\bigcup\limits_{B\in T_1^+}V_j^B$ and $V_{j1}^-=\bigcup\limits_{B\in T_1^-}V_j^B$. If $p_2=1$, then by the second statement of this lemma, we have $q_A=1$. So we may suppose $p_2\geqslant 2$ and $q_A>\binom{p_2}{\lfloor\frac{p_2}{2}\rfloor}$. By Lemma~\ref{lem-Kpq-distance-4}, there exist vertices $z,w\in V_i^A$ such that $\partial_F(z,w)\geqslant 4$. We know $\partial_D(z,w)\geqslant 2$. Since $N_D^+(z)\subseteq ([3]\setminus A)\cup V_{j1}^+\cup V_{j2}$ and $N_D^-(w)\subseteq A\cup V_{j1}^-\cup V_{j2}$, we get $N_D^+(z)\cap N_D^-(w)\subseteq V_{j2}$. We know $N_D^+(z)\cap N_D^-(w)\cap V_{j2}=N_F^+(z)\cap N_F^-(w)=\emptyset$ by $\partial_F(z,w)\geqslant 4$. Hence $N_D^+(z)\cap N_D^-(w)=\emptyset$. By Lemma~\ref{lem-neighbour-diamter-2}, we have $\partial_D(z,w)\neq 2$. Thus $\partial_D(z,w)\geqslant 3$. A contradiction. 

Take any $y\in V_j^B$ and any $z\in V_i^A$. Then $N_D^+(y)\subseteq \overline{B}\cup V_i$ and $N_D^-(z)\subseteq \overline{A}\cup V_j$. Note that $A\cap \overline{B}=\emptyset$, so we have $N_D^+(y)\cap N_D^-(w)=\emptyset$. By Lemma~\ref{lem-neighbour-diamter-2}, we have $\partial_D(y,z)\neq 2$. Since $\partial_D(y,z)\leqslant 2$, we have $\partial_D(y,z)=1$, which means $y\rightarrow z$. Hence $V_j^B\rightarrow V_i^A$. 
\end{proof}

Let $D$ be an orientation of $\K(3,p,q)$ with diameter two. By the above results, in $D$, all the orientations between $V_i$ and $V_j$ are known except for those between $V_i^A$ and $V_j^B$ with $A,B\subseteq [3]$ satisfying $A\setminus B\neq\emptyset$ and $B\setminus A\neq\emptyset$.



\section{Big part: structural analysis of orientations with diameter two}\label{sec-big-part}

In this section, we consider all strong orientations with diameter two according to the partitions of $V_2$ of $\K(3,p,q)$ with $p\geqslant 5$. Let $D$ be a strong orientation of $\K(3,p,q)$, and let 
\begin{center}
$\mathbb{H}=\{V_2^+, V_2^1, V_2^2, V_2^3, V_2^{12}, V_2^{13}, V_2^{23}, V_2^-\}$. 
\end{center}
Then $\mathbb{H}$ form a partition of $V_2$. We will divide it into cases according to the number of nonempty sets in $\mathbb{H}$.

\subsection{There is exactly one nonempty set in $\mathbb{H}$}\label{subsec-H=1}

\begin{lemma}\label{lem-H=1-V2+}
Let $3\leqslant p\leqslant q$, and let $D$ be a strong orientation of $\K(3,p,q)$. If $V_2=V_2^+$ or $V_2^-$, then $\diam(D)\geqslant 3$.
\end{lemma}

\begin{proof}
Suppose $V_2=V_2^+$. By Lemma~\ref{lem-diam2-ViA-VjB}, we have $V_1\rightarrow V_2\rightarrow V_3$. Since $p=|V_2|>2$, we take $y,w\in V_2$, and so we have $\partial_D(y,w)\geqslant 3$. A contradiction. When $V_2=V_2^-$, this case in $D$ is the same situation as $V_2=V_2^+$ in $\rev(D)$. 
\end{proof}

\begin{lemma}\label{lem-H=1-V2A}
Let $3\leqslant p\leqslant q$, and let $D$ be a strong orientation of $\K(3,p,q)$. If $V_2=V_2^A$ for some nonempty and proper subset $A$ of $[3]$, then $\diam(D)\geqslant 3$.
\end{lemma}

\begin{proof}
Suppose $V_2=V_2^A$. We have $V_{11}\rightarrow V_2\rightarrow V_{12}$ where $V_{11}=\{x_i\mid i\in A\}$, $V_{12}=\{x_i\mid i\in \overline{A}\}$, and $\overline{A}=[3]\setminus A$. We know $|A|\geqslant 2$ or $|\overline{A}|\geqslant 2$. The proofs for these two cases are similar, so we may assume $|A|\geqslant 2$. Then $|V_{11}|\geqslant 2$. If there exist vertices $x\in V_{11}$ and $z\in V_3$ such that $x\rightarrow z$, then $\partial_D(z,x)\geqslant 2$. Since $N_D^+(z)\subseteq V_1\cup V_2$ and $N_D^-(x)\subseteq V_3$, we have $N_D^+(z)\cap N_D^-(x)=\emptyset$. By Lemma~\ref{lem-neighbour-diamter-2}, we have $\partial_D(z,x)\neq 2$. This means $\partial_D(z,x)\geqslant 3$. Now suppose for any $x\in V_{11}$ and any $z\in V_3$ we have $z\rightarrow x$, i.e., $V_3\rightarrow V_{11}$. Take two vretices $x,w\in V_{11}$, then $\partial_D(x,w)\geqslant 2$. Since $N_D^+(x)\subseteq V_2$ and $N_D^-(w)\subseteq V_3$, we have $N_D^+(x)\cap N_D^-(w)=\emptyset$. By Lemma~\ref{lem-neighbour-diamter-2}, we have $\partial_D(x,w)\neq 2$. This means $\partial_D(x,w)\geqslant 3$. 
\end{proof}

By the above lemmas in this subsection, we have the following result. 

\begin{theorem}\label{thm-H=1}
Let $3\leqslant p\leqslant q$, and let $D$ be a strong orientation of $\K(3,p,q)$. If $|\mathbb{H}|=1$, then $\diam(D)\geqslant 3$.
\end{theorem}

\subsection{There are exactly two nonempty sets in $\mathbb{H}$}\label{subsec-H=2}

Let $3\leqslant p\leqslant q$, and let $D$ be a strong orientation of $\K(3,p,q)$ with diameter two. If there are exactly two nonempty sets in $\mathbb{H}$, by Lemma~\ref{lem-diam2-ViA-VjB}, $V_2\neq V_2^+\cup V_2^-$, hence $V_2=V_2^+\cup V_2^A$, $V_2=V_2^-\cup V_2^B$, or $V_2=V_2^A\cup V_2^B$ where $A$ and $B$ are nonempty proper subsets of $[3]$. The case $V_2=V_2^-\cup V_2^B$ in $D$ is the same situation as $V_2=V_2^+\cup V_2^A$ in $\rev(D)$ where $A=[3]\setminus B$. Hence we only need to consider two cases $V_2=V_2^+\cup V_2^A$ and $V_2=V_2^A\cup V_2^B$.

\begin{lemma}\label{lem-H=2-V2+V2A}
Let $3\leqslant p\leqslant q$, and let $D$ be a strong orientation of $\K(3,p,q)$ with diameter two. If $V_2=V_2^+\cup V_2^A$ for some nonempty and proper subset $A$ of $[3]$, then $q\leqslant 1+3\binom{p-1}{\lfloor\frac{p-1}{2}\rfloor}$.
\end{lemma}

\begin{proof}
We know $|A|=1$ or $|A|=2$. First we consider the case $|A|=2$, we may suppose $A=\{x_1,x_2\}$. Then $V_2=V_2^{+}\cup V_2^{12}$, $V_3=V_3^1\cup V_3^2\cup V_3^3\cup V_3^{13}\cup V_3^{23}\cup V_3^-$. We have $x_1\rightarrow V_2$ and $x_2\rightarrow V_2$. Since $\diam(V_3,x_1)\leqslant 2$ and $\diam(V_3,x_2)\leqslant 2$, we have $V_3\rightarrow x_1$ and $V_3\rightarrow x_2$. So $V_3=V_3^{3}\cup V_3^-$. Since $N_D^+(x_1)\cap N_D^-(x_2)=V_2\cap V_3=\emptyset$, by Lemma~\ref{lem-neighbour-diamter-2}, we have $\partial_D(x_1,x_2)\geqslant 3$, which is a contradiction. Hence we have $|A|=1$. 

We may assume $A=\{x_1\}$. Then $V_2=V_2^+\cup V_2^1$, $V_3=V_3^2\cup V_3^3\cup V_3^{12}\cup V_3^{23}\cup V_3^{13}\cup V_3^-$. We know $x_1\rightarrow V_2$, since $\diam(V_3,x_1)\leqslant 2$, then $V_3\rightarrow x_1$ and so $V_3^{12}=\emptyset$ and $V_3^{13}=\emptyset$. This means $V_3=V_3^2\cup V_3^3\cup V_3^{23} \cup V_3^-$. By Lemma~\ref{lem-diam2-ViA-VjB}, we get $V_2^+\rightarrow V_3$, $V_2\rightarrow V_3^-$ and $V_2^1\rightarrow V_3^-$. So $(V_2^1,V_3^2)$, $(V_2^1,V_3^3)$, $(V_2^1,V_3^{23})$ are mixed pairs. Hence $q=|V_3|\leqslant 1+3\binom{p-1}{\lfloor\frac{p-1}{2}\rfloor}$.
\end{proof}

\begin{lemma}\label{lem-H=2-V2AV2B-eq}
Let $3\leqslant p\leqslant q$, and let $D$ be a strong orientation of $\K(3,p,q)$ with diameter two. If $V_2=V_2^A\cup V_2^B$ for two nonempty proper subsets $A$ and $B$ of $[3]$ with $|A|=|B|$, then $q\leqslant \max\Bigl\{2+\binom{p-1}{\lfloor\frac{p-1}{2}\rfloor}+\binom{p}{\lfloor\frac{p}{2}\rfloor}, 1+2\binom{p-2}{\lfloor\frac{p-2}{2}\rfloor}+\binom{p}{\lfloor\frac{p}{2}\rfloor}\Bigr\}$.
\end{lemma}

\begin{proof}
The situation of $|A|=|B|=2$ in $D$ is the same as $|\overline{A}|=|\overline{B}|=1$ in $\rev(D)$, where $\overline{A}=[3]\setminus A$ and  $\overline{B}=[3]\setminus B$. So we may assume $|A|=|B|=1$ and suppose $A=\{x_1\}$ and $B=\{x_2\}$. Then $V_2=V_2^1\cup V_2^2$, $V_3=V_3^+\cup V_3^3\cup V_3^{12}\cup V_3^{13}\cup V_3^{23}\cup V_3^-$. We know $V_2\rightarrow x_3$, since $\diam(x_3,V_3)\leqslant 2$, then $x_3\rightarrow V_3$. Hence $V_3^{12}=\emptyset=V_3^-$, and so $V_3=V_3^+\cup V_3^3\cup V_3^{13}\cup V_3^{23}$. By Lemma~\ref{lem-diam2-ViA-VjB}, we have $V_3^+\rightarrow V_2$, $V_3^{13}\rightarrow V_2^1$ and $V_3^{23}\rightarrow V_2^2$. So $(V_2^1,V_3^{23})$, $(V_2^2,V_3^{13})$, $(V_2,V_3^3)$ are mixed pairs. Let $p_1=|V_2^1|$ and $p_2=|V_2^2|$. Then $1\leqslant p_1,~p_2\leqslant p-1$ and $p_1+p_2=p$. Hence $q=|V_3|\leqslant 1+\binom{p_1}{\lfloor\frac{p_1}{2}\rfloor}+\binom{p_2}{\lfloor\frac{p_2}{2}\rfloor}+\binom{p}{\lfloor\frac{p}{2}\rfloor}\leqslant \max\Bigl\{2+\binom{p-1}{\lfloor\frac{p-1}{2}\rfloor}+\binom{p}{\lfloor\frac{p}{2}\rfloor}, 1+2\binom{p-2}{\lfloor\frac{p-2}{2}\rfloor}+\binom{p}{\lfloor\frac{p}{2}\rfloor}\Bigr\}$.
\end{proof}

\begin{lemma}\label{lem-H=2-V2AV2B-neq}
Let $3\leqslant p\leqslant q$, and let $D$ be a strong orientation of $\K(3,p,q)$ with diameter two. If $V_2=V_2^A\cup V_2^B$ for two nonempty proper subsets $A$ and $B$ of $[3]$ with $|A|\neq|B|$, then $q\leqslant \max\Bigl\{4+2\binom{p-1}{\lfloor\frac{p-1}{2}\rfloor}, 2+4\binom{p-2}{\lfloor\frac{p-2}{2}\rfloor}, \binom{p+1}{\lfloor\frac{p+1}{2}\rfloor}-1\Bigr\}$.
\end{lemma}

\begin{proof}
We may assume $|A|=1$, $|B|=2$. There are two cases. 

First, we suppose $A\cap B=\emptyset$, we may assume $A=\{x_1\}$ and $B=\{x_2, x_3\}$. Then $V_2=V_2^1\cup V_2^{23}$, $V_3=V_3^+\cup V_3^2\cup V_3^3\cup V_3^{12}\cup V_3^{13}\cup V_3^-$. By Lemma~\ref{lem-diam2-ViA-VjB}, we have $V_3^+\rightarrow V_2\rightarrow V_3^-$, $V_3^{12}\cup V_3^{13}\rightarrow V_2^1$, $V_2^{23}\rightarrow V_3^2\cup V_3^3$. So $(V_2^1,V_3^2)$, $(V_2^1,V_3^3)$, $(V_2^{23},V_3^{12})$, $(V_2^{23},V_3^{13})$ are mixed pairs. Let $p_1=|V_2^1|$ and $p_2=|V_2^{23}|$. Then $1\leqslant p_1,~p_2\leqslant p-1$ and $p_1+p_2=p$. Hence we have $q=|V_3|\leqslant 1+\binom{p_2}{\lfloor\frac{p_2}{2}\rfloor}+\binom{p_2}{\lfloor\frac{p_2}{2}\rfloor}+\binom{p_1}{\lfloor\frac{p_1}{2}\rfloor}
+\binom{p_1}{\lfloor\frac{p_1}{2}\rfloor}+1=2+2\binom{p_1}{\lfloor\frac{p_1}{2}\rfloor}+2\binom{p_2}{\lfloor\frac{p_2}{2}\rfloor}\leqslant \max\Bigl\{4+2\binom{p-1}{\lfloor\frac{p-1}{2}\rfloor}, 2+4\binom{p-2}{\lfloor\frac{p-2}{2}\rfloor}\Bigr\}$.

Now, we suppose $A\cap B\neq\emptyset$, then $A\subseteq B$, and we may assume $A=\{x_1\}$ and $B=\{x_1, x_2\}$. Then 
\begin{center}
  $V_2=V_2^1\cup V_2^{12}$. 
\end{center}
We know $x_1\rightarrow V_2\rightarrow x_3$, since $\diam(V_3,x_1)\leqslant 2$ and $\diam(x_3,V_3)\leqslant 2$, then $x_3\rightarrow V_3\rightarrow x_1$. So 
\begin{center}
 $V_3=V_3^{23}\cup V_3^3$. 
\end{center}
Let 
\begin{center}
  $\mathbb{S}=\{N_D^+(z)\cap V_2\mid z\in V_3^{23}\}$,\\ 
  $\mathbb{T}=\{N_D^+(z)\cap V_2\mid z\in V_3^{3}\}$.~~
\end{center}
Since $\diam(V_3^3,V_3^3)\leqslant 2$ and $\diam(V_3^{23},V_3^{23})\leqslant 2$, then $\mathbb{S}$ and $\mathbb{T}$ are independent collections of $V_2$. We know $V_3^3\rightarrow x_2\rightarrow V_3^{23}$, since $\diam(V_3^{23},V_3^3)\leqslant 2$, we have $S\setminus T\neq\emptyset$ for any $S\in \mathbb{S}$ and any $T\in \mathbb{T}$. 

Take a new vertex $y\not\in V_2$, and let 
\begin{center}
  $\mathbb{T}^*=\{T\cup \{y\}\mid T\in \mathbb{T}\}$.
\end{center}
Then $\mathbb{T}^*$ is an independent collection of $V_2^*=V_2\cup \{y\}$. Note that $\mathbb{S}$ is also an independent collection of $V_2^*$. Take any $S\in \mathbb{S}$ and any $T\in \mathbb{T}$, we know $y\not\in S$, so $S\setminus (T\cup \{y\})=S\setminus T\neq\emptyset$ and $y\in (T\cup \{y\})\setminus S\neq\emptyset$. Hence $\mathbb{S}\cup \mathbb{T}^*$ is an independent collection of $V_2^*$. By Sperner's Lemma, we have 
\begin{center}
  $|\mathbb{S}\cup \mathbb{T}^*|\leqslant \binom{p+1}{\lfloor\frac{p+1}{2}\rfloor}$, 
\end{center}
and the equality holds if and only if all the subsets in $\mathbb{S}\cup \mathbb{T}^*$ have the same size $\lambda$, where $\lambda=\lfloor\frac{p+1}{2}\rfloor$ or $\lceil\frac{p+1}{2}\rceil$. We suppose the equality holds, then 
\begin{center}
  $\mathbb{S}=\binom{V_2}{\lambda}$ and $\mathbb{T}=\binom{V_2}{\lambda-1}$. 
\end{center}
If $|V_2^{12}|\leqslant \lambda-1$, then we can take $z\in V_3^3$ such that $V_2^{12}\subseteq N_D^+(z)\cap V_2\in \mathbb{T}$, and so $\partial_D(x_2,z)\geqslant 3$, which is a contradiction. If $|V_2^{12}|\geqslant \lambda$, then we can take $z\in V_3^{23}$ such that $V_2^{12}\supseteq N_D^+(z)\cap V_2\in \mathbb{S}$, and so $\partial_D(z,x_2)\geqslant 3$, which is a contradiction. Hence the equality doesn't hold, and we have $q=|V_3|=|V_3^3|+|V_3^{23}|=|\mathbb{S}|+|\mathbb{T}|=|\mathbb{S}|+|\mathbb{T}^*|=|\mathbb{S}\cup \mathbb{T}^*|<\binom{p+1}{\lfloor\frac{p+1}{2}\rfloor}$. 
\end{proof}

When $p\geqslant 5$, we have 
\begin{eqnarray*}
  1+3\binom{p-1}{\lfloor\frac{p-1}{2}\rfloor} & \leqslant & \binom{p+1}{\lfloor\frac{p+1}{2}\rfloor}-1, \\
  \max\Bigl\{2+\binom{p-1}{\lfloor\frac{p-1}{2}\rfloor}+\binom{p}{\lfloor\frac{p}{2}\rfloor}, 1+2\binom{p-2}{\lfloor\frac{p-2}{2}\rfloor}+\binom{p}{\lfloor\frac{p}{2}\rfloor}\Bigr\} & < & \binom{p+1}{\lfloor\frac{p+1}{2}\rfloor}-1, \\
  \max\Bigl\{4+2\binom{p-1}{\lfloor\frac{p-1}{2}\rfloor}, 2+4\binom{p-2}{\lfloor\frac{p-2}{2}\rfloor} \Bigr\} & < & \binom{p+1}{\lfloor\frac{p+1}{2}\rfloor}-1.
\end{eqnarray*}
By the above lemmas in this subsection, we have the following result. 

\begin{theorem}\label{thm-H=2}
Let $5\leqslant p\leqslant q$, and let $D$ be a strong orientation of $\K(3,p,q)$ with diameter two. If $|\mathbb{H}|=2$, then 
\begin{center}
  $q\leqslant \binom{p+1}{\lfloor\frac{p+1}{2}\rfloor}-1$.
\end{center}
\end{theorem}

\subsection{There are exactly three nonempty sets in $\mathbb{H}$}\label{subsec-H=3}

Let $3\leqslant p\leqslant q$, and let $D$ be a strong orientation of $\K(3,p,q)$ with diameter two. If there are exactly three nonempty sets in $\mathbb{H}$, then $V_2=V_2^+\cup V_2^A\cup V_2^-$, $V_2=V_2^+\cup V_2^A\cup V_2^B$, $V_2=V_2^A\cup V_2^B\cup V_2^-$ or $V_2=V_2^A\cup V_2^B\cup V_2^C$ where $A$, $B$ and $C$ are nonempty proper subsets of $[3]$. The case $V_2=V_2^A\cup V_2^B\cup V_2^-$ in $D$ is the same situation as $V_2=V_2^+\cup V_2^{\overline{A}}\cup V_2^{\overline{B}}$ in $\rev(D)$ where $\overline{A}=[3]\setminus A$ and $\overline{B}=[3]\setminus B$. Hence we only need to consider three cases.

\begin{lemma}\label{lem-H=3-V2+V2AV2-}
Let $3\leqslant p\leqslant q$, and let $D$ be a strong orientation of $\K(3,p,q)$ with diameter two. If $V_2=V_2^+\cup V_2^A\cup V_2^-$ for some nonempty and proper subset $A$ of $[3]$, then $q\leqslant 2+3\binom{p-2}{\lfloor\frac{p-2}{2}\rfloor}$.
\end{lemma}

\begin{proof}
The case $V_2=V_2^+\cup V_2^A\cup V_2^-$ in $D$ is the same situation as $V_2=V_2^+\cup V_2^{\overline{A}}\cup V_2^-$ in $\rev(D)$ where $\overline{A}=[3]\setminus A$, hence we may assume $|A|=1$. Suppose $A=\{x_1\}$. Then $V_2=V_2^+\cup V_2^1\cup V_2^-$, $V_3=V_3^{12}\cup V_3^{13}\cup V_3^2\cup V_3^3\cup V_3^{23}$. By Lemma~\ref{lem-diam2-ViA-VjB}, we have $V_1\rightarrow V_2^+\rightarrow V_3$, $V_3\rightarrow V_2^-\rightarrow V_1$, $V_3^{12}\cup V_3^{13}\rightarrow V_2^1$. So $|V_3^{12}|\leqslant 1$, $|V_3^{13}|\leqslant 1$, $(V_2^1,V_3^2)$, $(V_2^1,V_3^3)$, $(V_2^1,V_3^{23})$ are mixed pairs. Hence $q\leqslant 1+1+\binom{p-2}{\lfloor\frac{p-2}{2}\rfloor}+\binom{p-2}{\lfloor\frac{p-2}{2}\rfloor}
+\binom{p-2}{\lfloor\frac{p-2}{2}\rfloor}=2+3\binom{p-2}{\lfloor\frac{p-2}{2}\rfloor}$. 
\end{proof}

\begin{lemma}\label{lem-H=3-V2+V2AV2B-eq}
Let $3\leqslant p\leqslant q$, and let $D$ be a strong orientation of $\K(3,p,q)$ with diameter two. If $V_2=V_2^+\cup V_2^A\cup V_2^B$ for some nonempty and proper subsets $A$ and $B$ of $[3]$ with $|A|=|B|$, then $q\leqslant 2+2\binom{p-1}{\lfloor\frac{p-1}{2}\rfloor}$.
\end{lemma}

\begin{proof}
Let $p_1=|V_2^A|$ and $p_2=|V_2^B|$. Then $1\leqslant p_1,p_2\leqslant p-2$ and $p_1+p_2=p-1$. First, we assume $|A|=|B|=1$, and suppose $A=\{x_1\}$ and $B=\{x_2\}$. Then $V_2=V_2^+\cup V_2^1\cup V_2^2$ and $V_3=V_3^{12}\cup V_3^{13}\cup V_3^{23}\cup V_3^3\cup V_3^-$. By Lemma~\ref{lem-diam2-ViA-VjB}, we have $V_1\rightarrow V_2^+\rightarrow V_3$, $V_2\rightarrow V_3^-\rightarrow V_1$, $V_3^{12}\cup V_3^{13}\rightarrow V_2^1$ and $V_3^{12}\cup V_3^{23}\rightarrow V_2^2$. So $|V_3^{12}|\leqslant 1$, $|V_3^-|\leqslant 1$, $(V_2^1,V_3^{23})$, $(V_2^2,V_3^{13})$ and $(V_2^1\cup V_2^2,V_3^3)$ are mixed pairs. Hence $q\leqslant 1+\binom{p_2}{\lfloor\frac{p_2}{2}\rfloor}+\binom{p_1}{\lfloor\frac{p_1}{2}\rfloor}+\binom{p-1}{\lfloor\frac{p-1}{2}\rfloor}+1
\leqslant 2+\binom{p_1+p_2}{\lfloor\frac{p_1+p_2}{2}\rfloor}+\binom{p-1}{\lfloor\frac{p-1}{2}\rfloor}=
2+2\binom{p-1}{\lfloor\frac{p-1}{2}\rfloor}$. 

Now suppose $|A|=|B|=2$, we may assume $A=\{x_1,x_2\}$ and $B=\{x_2,x_3\}$. Then $V_2=V_2^+\cup V_2^{12}\cup V_2^{23}$. We know $x_2\rightarrow V_2$, since $\diam(V_3,x_2)\leqslant 2$, we have $V_3\rightarrow x_2$, and so $V_3=V_3^1\cup V_3^3\cup V_3^{13}\cup V_3^-$. By Lemma~\ref{lem-diam2-ViA-VjB}, we have $V_1\rightarrow V_2^+\rightarrow V_3$, $V_2\rightarrow V_3^-\rightarrow V_1$, $V_2^{12}\rightarrow V_3^1$ and $V_2^{23}\rightarrow V_3^3$. Then $|V_3^-|\leqslant 1$, $(V_2^{12},V_3^3)$, $(V_2^{23},V_3^1)$ and $(V_2^{12}\cup V_2^{23},V_3^{13})$ are mixed pairs. Hence $q\leqslant \binom{p_2}{\lfloor\frac{p_2}{2}\rfloor}+\binom{p_1}{\lfloor\frac{p_1}{2}\rfloor}+\binom{p-1}{\lfloor\frac{p-1}{2}\rfloor}+1
\leqslant 1+\binom{p_1+p_2}{\lfloor\frac{p_1+p_2}{2}\rfloor}+\binom{p-1}{\lfloor\frac{p-1}{2}\rfloor}=
1+2\binom{p-1}{\lfloor\frac{p-1}{2}\rfloor}$. 
\end{proof}

\begin{lemma}\label{lem-H=3-V2+V2AV2B-neq}
Let $3\leqslant p\leqslant q$, and let $D$ be a strong orientation of $\K(3,p,q)$ with diameter two. If $V_2=V_2^+\cup V_2^A\cup V_2^B$ for some nonempty and proper subsets $A$ and $B$ of $[3]$ with $|A|\neq|B|$, then $q\leqslant 1+\binom{p-2}{\lfloor\frac{p-2}{2}\rfloor}+2\binom{p-1}{\lfloor\frac{p-1}{2}\rfloor}$.
\end{lemma}

\begin{proof}
We may assume $|A|=1$ and $|B|=2$, so there are two cases. Let $p_1=|V_2^A|$ and $p_2=|V_2^B|$. Then $1\leqslant p_1,p_2\leqslant p-2$ and $p_1+p_2=p-1$. 

First, we let $A\cap B=\emptyset$. We may suppose $A=\{x_1\}$ and $B=\{x_2,x_3\}$. Then $V_2=V_2^+\cup V_2^1\cup V_2^{23}$ and $V_3=V_3^{12}\cup V_3^{13}\cup V_3^2\cup V_3^{3}\cup V_3^-$. By Lemma~\ref{lem-diam2-ViA-VjB}, we have $V_1\rightarrow V_2^+\rightarrow V_3$, $V_2\rightarrow V_3^-\rightarrow V_1$, $V_3^{12}\cup V_3^{13}\rightarrow V_2^1$ and $V_2^{23}\rightarrow V_3^2\cup V_3^3$. So $|V_3^-|\leqslant 1$, $(V_2^1,V_3^2)$, $(V_2^1,V_3^3)$, $(V_2^{23},V_3^{12})$ and $(V_2^{23},V_3^{13})$ are mixed pairs. Hence $q\leqslant \binom{p_2}{\lfloor\frac{p_2}{2}\rfloor}+\binom{p_2}{\lfloor\frac{p_2}{2}\rfloor}+\binom{p_1}{\lfloor\frac{p_1}{2}\rfloor}
+\binom{p_1}{\lfloor\frac{p_1}{2}\rfloor}+1
\leqslant 1+2\binom{p_1+p_2}{\lfloor\frac{p_1+p_2}{2}\rfloor}=1+2\binom{p-1}{\lfloor\frac{p-1}{2}\rfloor}$. 

Now we let $A\cap B\neq\emptyset$. We may suppose $A=\{x_1\}$ and $B=\{x_1,x_2\}$. Then $V_2=V_2^+\cup V_2^1\cup V_2^{12}$. We know $x_1\rightarrow V_2$, since $\diam(V_3,x_1)\leqslant 2$, we have $V_3\rightarrow x_1$, and so $V_3=V_3^2\cup V_3^{23}\cup V_3^3\cup V_3^-$. By Lemma~\ref{lem-diam2-ViA-VjB}, we have $V_1\rightarrow V_2^+\rightarrow V_3$, $V_2\rightarrow V_3^-\rightarrow V_1$, $V_2^{12}\rightarrow V_3^2$. So $|V_3^-|\leqslant 1$, $(V_2^1,V_3^2)$, $(V_2^1\cup V_2^{12},V_3^{23})$ and $(V_2^1\cup V_2^{12},V_3^3)$ are mixed pairs. Hence $q\leqslant \binom{p_1}{\lfloor\frac{p_1}{2}\rfloor}+\binom{p_1+p_2}{\lfloor\frac{p_1+p_2}{2}\rfloor}+\binom{p_1+p_2}{\lfloor\frac{p_1+p_2}{2}\rfloor}+1
\leqslant 1+\binom{p-2}{\lfloor\frac{p-2}{2}\rfloor}+2\binom{p-1}{\lfloor\frac{p-1}{2}\rfloor}$. 
\end{proof}

\begin{lemma}\label{lem-H=3-V2AV2BV2C}
Let $4\leqslant p\leqslant q$, and let $D$ be a strong orientation of $\K(3,p,q)$ with diameter two. If $V_2=V_2^A\cup V_2^B\cup V_2^C$ for some nonempty and proper subsets $A$, $B$ and $C$ of $[3]$, then $q\leqslant 1+\binom{p-3}{\lfloor\frac{p-3}{2}\rfloor}+\binom{p-1}{\lfloor\frac{p-1}{2}\rfloor}+\binom{p}{\lfloor\frac{p}{2}\rfloor}$.
\end{lemma}

\begin{proof}
The case $V_2=V_2^A\cup V_2^B\cup V_2^C$ in $D$ is the same situation as $V_2=V_2^{\overline{A}}\cup V_2^{\overline{B}}\cup V_2^{\overline{C}}$ in $\rev(D)$ where $\overline{A}=[3]\setminus A$, $\overline{B}=[3]\setminus B$ and $\overline{C}=[3]\setminus C$. We may assume $|A|\leqslant |B|\leqslant |C|$, then we only need to consider $(|A|,|B|,|C|)=(1,1,1)$ or $(1,1,2)$. Let $p_1=|V_2^A|$, $p_2=|V_2^B|$ and $p_3=|V_2^C|$. Then $1\leqslant p_1,p_2,p_3\leqslant p-2$ and $p_1+p_2+p_3=p$.

First, let $(|A|,|B|,|C|)=(1,1,1)$, we may assume $A=\{x_1\}$, $B=\{x_2\}$ and $C=\{x_3\}$. Then $V_2=V_2^1\cup V_2^2\cup V_2^3$ and $V_3=V_3^+\cup V_3^{12}\cup V_3^{13}\cup V_3^{23}\cup V_3^-$. By Lemma~\ref{lem-diam2-ViA-VjB}, we have $V_1\rightarrow V_3^+\rightarrow V_2$, $V_2\rightarrow V_3^-\rightarrow V_1$, $V_3^{12}\rightarrow V_2^1\cup V_2^2$, $V_3^{13}\rightarrow V_2^1\cup V_2^3$ and $V_3^{23}\rightarrow V_2^2\cup V_2^3$. So $|V_3^+|\leqslant 1$, $|V_3^-|\leqslant 1$, $(V_2^1,V_3^{23})$, $(V_2^2,V_3^{13})$ and $(V_2^3,V_3^{12})$ are mixed pairs. Hence $q\leqslant 1+\binom{p_1}{\lfloor\frac{p_1}{2}\rfloor}+\binom{p_2}{\lfloor\frac{p_2}{2}\rfloor}+\binom{p_3}{\lfloor\frac{p_3}{2}\rfloor}+1\leqslant 2+\binom{p}{\lfloor\frac{p}{2}\rfloor}$. 

Now, we let $(|A|,|B|,|C|)=(1,1,2)$, then $A\subseteq C$ or $B\subseteq C$. If $A\subseteq C$ and $B\cup C=V_1$, we may assume $A=\{x_1\}$, $B=\{x_2\}$ and $C=\{x_1,x_3\}$, then $V_2=V_2^1\cup V_2^2\cup V_2^{13}$ and $V_3=V_3^+\cup V_3^3\cup V_3^{12}\cup V_3^{23}\cup V_3^-$. By Lemma~\ref{lem-diam2-ViA-VjB}, we know $V_1\rightarrow V_3^+\rightarrow V_2$, $V_2\rightarrow V_3^-\rightarrow V_1$, $V_3^{12}\rightarrow V_2^1\cup V_2^2$, $V_3^{23}\rightarrow V_2^2$ and $V_2^{13}\rightarrow V_3^3$. So $|V_3^+|\leqslant 1$, $|V_3^-|\leqslant 1$, $(V_2^1\cup V_2^2,V_3^3)$, $(V_2^{13}, V_3^{12})$ and $(V_2^1\cup V_2^{13}, V_3^{23})$ are mixed pairs. Hence $q\leqslant 1+\binom{p_1+ p_2}{\lfloor\frac{p_1+p_2}{2}\rfloor}+\binom{p_3}{\lfloor\frac{p_3}{2}\rfloor}+\binom{p_1+p_3}{\lfloor\frac{p_1+p_3}{2}\rfloor}+1\leqslant 2+\binom{p}{\lfloor\frac{p}{2}\rfloor}+\binom{p-1}{\lfloor\frac{p-1}{2}\rfloor}$. 

If $A\cup B=C$, we may assume $A=\{x_1\}$, $B=\{x_2\}$ and $C=\{x_1,x_2\}$. Then $V_2=V_2^1\cup V_2^2\cup V_2^{12}$. We know $V_2\rightarrow x_3$, since $\diam(x_3,V_3)\leqslant 2$, then $x_3\rightarrow V_3$ and so $V_3=V_3^+\cup V_3^{13}\cup V_3^{23}\cup V_3^3$. By Lemma~\ref{lem-diam2-ViA-VjB}, we have $V_1\rightarrow V_3^+\rightarrow V_2$, $V_3^{13}\rightarrow V_2^1$ and $V_3^{23}\rightarrow V_2^2$. So $|V_3^+|\leqslant 1$, $(V_2^2\cup V_2^{12}, V_3^{13})$, $(V_2^1\cup V_2^{12}, V_3^{23})$ and $(V_2,V_3^3)$ are mixed pairs. Hence $q\leqslant 1+\binom{p_2+p_3}{\lfloor\frac{p_2+p_3}{2}\rfloor}+\binom{p_1+p_3}{\lfloor\frac{p_1+p_3}{2}\rfloor}+\binom{p}{\lfloor\frac{p}{2}\rfloor}$. 

When $p_1=p_2=1$, since $\diam(V_3^{13},x_1)\leqslant 2$ and $\diam(V_3^{23},x_2)\leqslant 2$, we have $V_3^{13}\rightarrow V_2^2$ and $V_3^{23}\rightarrow V_2^1$. So $(V_2^{12}, V_3^{13})$, $(V_2^{12}, V_3^{23})$ and $(V_2,V_3^3)$ are mixed pairs. Hence $q\leqslant 1+\binom{p_3}{\lfloor\frac{p_3}{2}\rfloor}+\binom{p_3}{\lfloor\frac{p_3}{2}\rfloor}+\binom{p}{\lfloor\frac{p}{2}\rfloor}
=1+2\binom{p-2}{\lfloor\frac{p-2}{2}\rfloor}+\binom{p}{\lfloor\frac{p}{2}\rfloor}$.

When $p_1=1$ and $p_2\geqslant 2$ (or, $p_1\geqslant 2$ and $p_2=1$), by the same argument as in the last paragraph, we have $V_3^{23}\rightarrow V_2^1$ (or $V_3^{13}\rightarrow V_2^2$). So $(V_2^2\cup V_2^{12}, V_3^{13})$, $(V_2^{12}, V_3^{23})$ and $(V_2,V_3^3)$ (or, $(V_2^{12}, V_3^{13})$, $(V_2^1\cup V_2^{12}, V_3^{23})$ and $(V_2,V_3^3)$) are mixed pairs. Hence $q\leqslant 1+\binom{p_2+p_3}{\lfloor\frac{p_2+p_3}{2}\rfloor}+\binom{p_3}{\lfloor\frac{p_3}{2}\rfloor}+\binom{p}{\lfloor\frac{p}{2}\rfloor}\leqslant 1+\binom{p-1}{\lfloor\frac{p-1}{2}\rfloor}+\binom{p-3}{\lfloor\frac{p-3}{2}\rfloor}+\binom{p}{\lfloor\frac{p}{2}\rfloor}$. 

When $p_1\geqslant 2$ and $p_2\geqslant 2$, we have $q\leqslant 1+\binom{p_2+p_3}{\lfloor\frac{p_2+p_3}{2}\rfloor}+\binom{p_1+p_3}{\lfloor\frac{p_1+p_3}{2}\rfloor}+\binom{p}{\lfloor\frac{p}{2}\rfloor}\leqslant 1+2\binom{p-2}{\lfloor\frac{p-2}{2}\rfloor}+\binom{p}{\lfloor\frac{p}{2}\rfloor}$.
\end{proof}

When $p\geqslant 4$ and $1\leqslant m\leqslant n$, we have 
\begin{eqnarray*}
  3\binom{p-1}{\lfloor\frac{p-1}{2}\rfloor} & < & 2\binom{p}{\lfloor\frac{p}{2}\rfloor},\\ 
  \binom{m+1}{\lfloor\frac{m+1}{2}\rfloor}-\binom{m}{\lfloor\frac{m}{2}\rfloor} & \leqslant & \binom{n+1}{\lfloor\frac{n+1}{2}\rfloor}-\binom{n}{\lfloor\frac{n}{2}\rfloor}.
\end{eqnarray*}
So we have
\begin{eqnarray*}
  2+3\binom{p-2}{\lfloor\frac{p-2}{2}\rfloor} &<& 2+2\binom{p-1}{\lfloor\frac{p-1}{2}\rfloor}, \\
  2+2\binom{p-1}{\lfloor\frac{p-1}{2}\rfloor} &<& 1+\binom{p-2}{\lfloor\frac{p-2}{2}\rfloor}+2\binom{p-1}{\lfloor\frac{p-1}{2}\rfloor}, \\
  1+\binom{p-2}{\lfloor\frac{p-2}{2}\rfloor}+2\binom{p-1}{\lfloor\frac{p-1}{2}\rfloor} &\leqslant & 1+\binom{p-3}{\lfloor\frac{p-3}{2}\rfloor}+\binom{p-1}{\lfloor\frac{p-1}{2}\rfloor}+\binom{p}{\lfloor\frac{p}{2}\rfloor}.
\end{eqnarray*}
By the above lemmas in this subsection, we have the following result. 

\begin{theorem}\label{thm-H=3}
Let $4\leqslant p\leqslant q$, and let $D$ be a strong orientation of $\K(3,p,q)$ with diameter two. If $|\mathbb{H}|=3$, then 
\begin{center}
  $q\leqslant 1+\binom{p-3}{\lfloor\frac{p-3}{2}\rfloor}+\binom{p-1}{\lfloor\frac{p-1}{2}\rfloor}+\binom{p}{\lfloor\frac{p}{2}\rfloor}$.
\end{center}
\end{theorem}

\subsection{There are exactly four nonempty sets in $\mathbb{H}$}\label{subsec-H=4}

Let $4\leqslant p\leqslant q$, and let $D$ be a strong orientation of $\K(3,p,q)$ with diameter two. If there are exactly four nonempty sets in $\mathbb{H}$, then $V_2=V_2^+\cup V_2^A\cup V_2^B\cup V_2^-$, $V_2=V_2^+\cup V_2^A\cup V_2^B\cup V_2^C$, $V_2=V_2^A\cup V_2^B\cup V_2^C\cup V_2^-$ or $V_2=V_2^A\cup V_2^B\cup V_2^C\cup V_2^E$ where $A$, $B$, $C$ and $E$ are nonempty proper subsets of $[3]$. The case $V_2=V_2^A\cup V_2^B\cup V_2^C\cup V_2^-$ in $D$ is the same situation as $V_2=V_2^+\cup V_2^{\overline{A}}\cup V_2^{\overline{B}}\cup V_2^{\overline{C}}$ in $\rev(D)$ where $\overline{A}=[3]\setminus A$, $\overline{B}=[3]\setminus B$ and $\overline{C}=[3]\setminus C$. Hence we only need to consider three cases.

\begin{lemma}\label{lem-H=4-V2+V2AV2BV2-}
Let $4\leqslant p\leqslant q$, and let $D$ be a strong orientation of $\K(3,p,q)$ with diameter two. If $V_2=V_2^+\cup V_2^A\cup V_2^B\cup V_2^-$ for some nonempty and proper subsets $A$ and $B$ of $[3]$, then $q\leqslant 3\binom{p-2}{\lfloor\frac{p-2}{2}\rfloor}$.
\end{lemma}

\begin{proof}
Let $p_1=|V_2^A|$ and $p_2=|V_2^B|$. Then $1\leqslant p_1,p_2\leqslant p-3$ and $p_1+p_2=p-2$. We first consider the case $|A|=|B|$, the situation of $V_2=V_2^+\cup V_2^A\cup V_2^B\cup V_2^-$ is the same as $V_2=V_2^+\cup V_2^{\overline{A}}\cup V_2^{\overline{B}}\cup V_2^-$ in $\rev(D)$ where $\overline{A}=[3]\setminus A$ and $\overline{B}=[3]\setminus B$, so may suppose $|A|=|B|=1$ and let $A=\{x_1\}$ and $B=\{x_2\}$. Then $V_2=V_2^+\cup V_2^1\cup V_2^2\cup V_2^-$ and $V_3=V_3^{12}\cup V_3^{13}\cup V_3^{23}\cup V_3^3$. By Lemma~\ref{lem-diam2-ViA-VjB}, we have $V_1\rightarrow V_2^+\rightarrow V_3$, $V_3\rightarrow V_2^-\rightarrow V_1$, $V_3^{12}\cup V_3^{13}\rightarrow V_2^1$ and $V_3^{12}\cup V_3^{23}\rightarrow V_2^2$. So $|V_3^{12}|\leqslant 1$, $(V_2^2,V_3^{13})$, $(V_2^1,V_3^{23})$ and $(V_2^1\cup V_2^2,V_3^3)$ are mixed pairs. Hence $q\leqslant 1+\binom{p_2}{\lfloor\frac{p_2}{2}\rfloor}+\binom{p_1}{\lfloor\frac{p_1}{2}\rfloor}+\binom{p_1+p_2}{\lfloor\frac{p_1+p_2}{2}\rfloor}\leqslant 1+2\binom{p_1+p_2}{\lfloor\frac{p_1+p_2}{2}\rfloor}=1+2\binom{p-2}{\lfloor\frac{p-2}{2}\rfloor}$. 

Now we consider the case $|A|\neq |B|$. We may suppose $|A|=1$ and $|B|=2$. If $A\subseteq B$, we may let $A=\{x_1\}$ and $B=\{x_1,x_2\}$. Then $V_2=V_2^+\cup V_2^1\cup V_2^{12}\cup V_2^-$ and $V_3=V_3^{13}\cup V_3^2\cup V_3^{23}\cup V_3^3$. By Lemma~\ref{lem-diam2-ViA-VjB}, we have $V_1\rightarrow V_2^+\rightarrow V_3$, $V_3\rightarrow V_2^-\rightarrow V_1$, $V_3^{13}\rightarrow V_2^1$ and $V_2^{12}\rightarrow V_3^2$. So $(V_2^{12},V_3^{13})$, $(V_2^1,V_3^2)$, $(V_2^1\cup V_2^{12},V_3^{23})$ and $(V_2^1\cup V_2^{12},V_3^3)$ are mixed pairs. Hence $q\leqslant \binom{p_2}{\lfloor\frac{p_2}{2}\rfloor}+\binom{p_1}{\lfloor\frac{p_1}{2}\rfloor}+\binom{p_1+p_2}{\lfloor\frac{p_1+p_2}{2}\rfloor}
+\binom{p_1+p_2}{\lfloor\frac{p_1+p_2}{2}\rfloor}\leqslant 3\binom{p_1+p_2}{\lfloor\frac{p_1+p_2}{2}\rfloor}=3\binom{p-2}{\lfloor\frac{p-2}{2}\rfloor}$. 

If $A\cap B=\emptyset$, we may assume $A=\{x_1\}$ and $B=\{x_2,x_3\}$. Then $V_2=V_2^+\cup V_2^1\cup V_2^{23}\cup V_2^-$ and $V_3=V_3^{12}\cup V_3^{13}\cup V_3^2\cup V_3^3$. By Lemma~\ref{lem-diam2-ViA-VjB}, we have $V_1\rightarrow V_2^+\rightarrow V_3$, $V_3\rightarrow V_2^-\rightarrow V_1$, $V_3^{12}\cup V_3^{13}\rightarrow V_2^1$ and $V_2^{23}\rightarrow V_3^2\cup V_3^3$. So $(V_2^{23},V_3^{12})$, $(V_2^{23},V_3^{13})$, $(V_2^1,V_3^2)$ and $(V_2^1,V_3^3)$ are mixed pairs. Hence $q\leqslant \binom{p_2}{\lfloor\frac{p_2}{2}\rfloor}+\binom{p_2}{\lfloor\frac{p_2}{2}\rfloor}+\binom{p_1}{\lfloor\frac{p_1}{2}\rfloor}
+\binom{p_1}{\lfloor\frac{p_1}{2}\rfloor}\leqslant 2\binom{p_1+p_2}{\lfloor\frac{p_1+p_2}{2}\rfloor}=2\binom{p-2}{\lfloor\frac{p-2}{2}\rfloor}$. 
\end{proof}

\begin{lemma}\label{lem-H=4-V2+V2AV2BV2C}
Let $4\leqslant p\leqslant q$, and let $D$ be a strong orientation of $\K(3,p,q)$ with diameter two. If $V_2=V_2^+\cup V_2^A\cup V_2^B\cup V_2^C$ for some nonempty and proper subsets $A$, $B$ and $C$ of $[3]$, then $q\leqslant 1+2\binom{p-2}{\lfloor\frac{p-2}{2}\rfloor}+\binom{p-1}{\lfloor\frac{p-1}{2}\rfloor}$.
\end{lemma}

\begin{proof}
Let $p_1=|V_2^A|$, $p_2=|V_2^B|$ and $p_3=|V_2^C|$. Then $1\leqslant p_1,p_2,p_3\leqslant p-3$ and $p_1+p_2+p_3=p-1$. We may assume $|A|\leqslant |B|\leqslant |C|$, then $(|A|,|B|,|C|)=(1,1,1)$, $(1,1,2)$, $(1,2,2)$ or $(2,2,2)$. 

Let $(|A|,|B|,|C|)=(1,1,1)$, we may suppose $A=\{x_1\}$, $B=\{x_2\}$ and $C=\{x_3\}$. Then $V_2=V_2^+\cup V_2^1\cup V_2^2\cup V_2^3$ and $V_3=V_3^{12}\cup V_3^{13}\cup V_3^{23}\cup V_3^-$. By Lemma~\ref{lem-diam2-ViA-VjB}, we have $V_1\rightarrow V_2^+\rightarrow V_3$, $V_2\rightarrow V_3^-\rightarrow V_1$, $V_3^{12}\rightarrow V_2^1\cup V_2^2$, $V_3^{13}\rightarrow V_2^1\cup V_2^3$, $V_3^{23}\rightarrow V_2^2\cup V_2^3$. So $|V_3^-|\leqslant 1$, $(V_2^3,V_3^{12})$, $(V_2^2,V_3^{13})$ and $(V_2^1,V_3^{23})$ are mixed pairs. Hence $q\leqslant \binom{p_3}{\lfloor\frac{p_3}{2}\rfloor}+\binom{p_2}{\lfloor\frac{p_2}{2}\rfloor}+\binom{p_1}{\lfloor\frac{p_1}{2}\rfloor}+1\leqslant 1+\binom{p_1+p_2+p_3}{\lfloor\frac{p_1+p_2+p_3}{2}\rfloor}=1+\binom{p-1}{\lfloor\frac{p-1}{2}\rfloor}$. 

Let $(|A|,|B|,|C|)=(1,1,2)$. Then $A\subseteq C$ or $B\subseteq C$. If $A\subseteq C$ and $B\cup C=V_1$, we may suppose $A=\{x_1\}$, $B=\{x_3\}$ and $C=\{x_1,x_2\}$. Then $V_2=V_2^+\cup V_2^1\cup V_2^{12}\cup V_2^3$ and $V_3=V_3^{13}\cup V_3^2\cup V_3^{23}\cup V_3^-$. By Lemma~\ref{lem-diam2-ViA-VjB}, we have $V_1\rightarrow V_2^+\rightarrow V_3$, $V_2\rightarrow V_3^-\rightarrow V_1$, $V_3^{13}\rightarrow V_2^1\cup V_2^3$, $V_3^{23}\rightarrow V_2^3$, $V_2^{12}\rightarrow V_3^2$. So $|V_3^-|\leqslant 1$, $(V_2^{12},V_3^{13})$, $(V_2^1\cup V_2^3,V_3^2)$ and $(V_2^1\cup V_2^{12},V_3^{23})$ are mixed pairs. Hence $q\leqslant \binom{p_2}{\lfloor\frac{p_2}{2}\rfloor}+\binom{p_1+p_3}{\lfloor\frac{p_1+p_3}{2}\rfloor}+\binom{p_1+p_2}{\lfloor\frac{p_1+p_2}{2}\rfloor}+1
\leqslant 1+\binom{p-2}{\lfloor\frac{p-2}{2}\rfloor}+\binom{p-1}{\lfloor\frac{p-1}{2}\rfloor}$. 

If $A\cup B=C$, we may suppose $A=\{x_1\}$, $B=\{x_2\}$ and $C=\{x_1,x_2\}$. Then $V_2=V_2^+\cup V_2^1\cup V_2^2\cup V_2^{12}$ and $V_3=V_3^{13}\cup V_3^{23}\cup V_3^3\cup V_3^-$. By Lemma~\ref{lem-diam2-ViA-VjB}, we have $V_1\rightarrow V_2^+\rightarrow V_3$, $V_2\rightarrow V_3^-\rightarrow V_1$, $V_3^{13}\rightarrow V_2^1$, $V_3^{23}\rightarrow V_2^2$. So $|V_3^-|\leqslant 1$, $(V_2^2\cup V_2^{12},V_3^{13})$, $(V_2^1\cup V_2^{12},V_3^{23})$ and $(V_2^1\cup V_2^2\cup V_2^{12},V_3^3)$ are mixed pairs. Hence $q\leqslant \binom{p_2+p_3}{\lfloor\frac{p_2+p_3}{2}\rfloor}+\binom{p_1+p_3}{\lfloor\frac{p_1+p_3}{2}\rfloor}+
\binom{p_1+p_2+p_3}{\lfloor\frac{p_1+p_2+p_3}{2}\rfloor}+1
\leqslant 1+2\binom{p-2}{\lfloor\frac{p-2}{2}\rfloor}+\binom{p-1}{\lfloor\frac{p-1}{2}\rfloor}$. 

Let $(|A|,|B|,|C|)=(1,2,2)$. Then $A\subseteq B\cup C=V_1$. If $A=B\cap C$, we may suppose $A=\{x_2\}$, $B=\{x_1,x_2\}$ and $C=\{x_2,x_3\}$. Then $V_2=V_2^+\cup V_2^{12}\cup V_2^2\cup V_2^{23}$ and $V_3=V_3^1\cup V_3^{13}\cup V_3^3\cup V_3^-$. By Lemma~\ref{lem-diam2-ViA-VjB}, we have $V_1\rightarrow V_2^+\rightarrow V_3$, $V_2\rightarrow V_3^-\rightarrow V_1$, $V_2^{12}\rightarrow V_3^1$, $V_2^{23}\rightarrow V_3^3$. So $|V_3^-|\leqslant 1$, $(V_2^2\cup V_2^{23},V_3^1)$, $(V_2^{12}\cup V_2^2\cup V_2^{23},V_3^{13})$ and $(V_2^{12}\cup V_2^2,V_3^3)$ are mixed pairs. Hence $q\leqslant \binom{p_2+p_3}{\lfloor\frac{p_2+p_3}{2}\rfloor}+\binom{p_1+p_2+p_3}{\lfloor\frac{p_1+p_2+p_3}{2}\rfloor}+
\binom{p_1+p_2}{\lfloor\frac{p_1+p_2}{2}\rfloor}+1\leqslant 1+2\binom{p-2}{\lfloor\frac{p-2}{2}\rfloor}+\binom{p-1}{\lfloor\frac{p-1}{2}\rfloor}$. 

If $A\neq B\cap C$, we may suppose $A=B\setminus C$, and let $A=\{x_1\}$, $B=\{x_1,x_2\}$ and $C=\{x_2,x_3\}$. Then $V_2=V_2^+\cup V_2^1\cup V_2^{12}\cup V_2^{23}$ and $V_3=V_3^{13}\cup V_3^2\cup V_3^3\cup V_3^-$. By Lemma~\ref{lem-diam2-ViA-VjB}, we have $V_1\rightarrow V_2^+\rightarrow V_3$, $V_2\rightarrow V_3^-\rightarrow V_1$, $V_2^{12}\rightarrow V_3^2$, $V_2^{23}\rightarrow V_3^2\cup V_3^3$, $V_3^{13}\rightarrow V_2^1$. So $|V_3^-|\leqslant 1$, $(V_2^{12}\cup V_2^{23},V_3^{13})$, $(V_2^1,V_3^2)$ and $(V_2^1\cup V_2^{12},V_3^3)$ are mixed pairs. Hence $q\leqslant \binom{p_2+p_3}{\lfloor\frac{p_2+p_3}{2}\rfloor}+\binom{p_1}{\lfloor\frac{p_1}{2}\rfloor}+
\binom{p_1+p_2}{\lfloor\frac{p_1+p_2}{2}\rfloor}+1\leqslant \binom{p_1+p_2+p_3}{\lfloor\frac{p_1+p_2+p_3}{2}\rfloor}+
\binom{p_1+p_2}{\lfloor\frac{p_1+p_2}{2}\rfloor}+1\leqslant \binom{p-1}{\lfloor\frac{p-1}{2}\rfloor}+\binom{p-2}{\lfloor\frac{p-2}{2}\rfloor}+1$. 

Let $(|A|,|B|,|C|)=(2,2,2)$, we may suppose $A=\{x_1,x_2\}$, $B=\{x_1,x_3\}$ and $C=\{x_2, x_3\}$. Then $V_2=V_2^+\cup V_2^{12}\cup V_2^{13}\cup V_2^{23}$ and $V_3=V_3^1\cup V_3^2\cup V_3^3\cup V_3^-$. By Lemma~\ref{lem-diam2-ViA-VjB}, we have $V_1\rightarrow V_2^+\rightarrow V_3$, $V_2\rightarrow V_3^-\rightarrow V_1$, $V_2^{12}\rightarrow V_3^1\cup V_3^2$, $V_2^{13}\rightarrow V_3^1\cup V_3^3$, $V_2^{23}\rightarrow V_3^2\cup V_3^3$. So $|V_3^-|\leqslant 1$, $(V_2^{23},V_3^1)$, $(V_2^{13},V_3^2)$ and $(V_2^{12},V_3^3)$ are mixed pairs. Hence $q\leqslant \binom{p_3}{\lfloor\frac{p_3}{2}\rfloor}+\binom{p_2}{\lfloor\frac{p_2}{2}\rfloor}+\binom{p_1}{\lfloor\frac{p_1}{2}\rfloor}+1\leqslant 1+\binom{p_1+p_2+p_3}{\lfloor\frac{p_1+p_2+p_3}{2}\rfloor}=1+\binom{p-1}{\lfloor\frac{p-1}{2}\rfloor}$. 
\end{proof}

\begin{lemma}\label{lem-H=4-V2AV2BV2CV2E}
Let $4\leqslant p\leqslant q$, and let $D$ be a strong orientation of $\K(3,p,q)$ with diameter two. If $V_2=V_2^A\cup V_2^B\cup V_2^C\cup V_2^E$ for some nonempty and proper subsets $A$, $B$, $C$ and $E$ of $[3]$, then $q\leqslant 2+2\binom{p-1}{\lfloor\frac{p-1}{2}\rfloor}$.
\end{lemma}

\begin{proof}
The situation of $V_2=V_2^A\cup V_2^B\cup V_2^C\cup V_2^E$ in $D$ is the same as $V_2=V_2^{\overline{A}}\cup V_2^{\overline{B}}\cup V_2^{\overline{C}}\cup V_2^{\overline{E}}$ in $\rev(D)$ where $\overline{A}=[3]\setminus A$, $\overline{B}=[3]\setminus B$, $\overline{C}=[3]\setminus C$ and $\overline{E}=[3]\setminus E$. We may suppose $|A|\leqslant |B|\leqslant |C|\leqslant |E|$, and so we only need to consider $(|A|,|B|,|C|,|E|)=(1,1,1,2)$ or $(1,1,2,2)$. Let $p_1=|V_2^A|$, $p_2=|V_2^B|$, $p_3=|V_2^C|$ and $p_4=|V_2^E|$. Then $1\leqslant p_1,p_2,p_3,p_4\leqslant p-3$ and $p_1+p_2+p_3+p_4=p$. 

Let $(|A|,|B|,|C|,|E|)=(1,1,1,2)$. Then $A\cup B\cup C=V_1$, we may assume $A=\{x_1\}$, $B=\{x_2\}$, $C=\{x_3\}$ and $E=\{x_1,x_2\}$. We have $V_2=V_2^1\cup V_2^{12}\cup V_2^2\cup V_2^3$ and $V_3=V_3^+\cup V_3^{13}\cup V_3^{23}\cup V_3^-$. By Lemma~\ref{lem-diam2-ViA-VjB}, we have $V_1\rightarrow V_3^+\rightarrow V_2$, $V_2\rightarrow V_3^-\rightarrow V_1$, $V_3^{13}\rightarrow V_2^1$, $V_3^{23}\rightarrow V_2^2$ and $V_3^{13}\cup V_3^{23}\rightarrow V_2^3$. So $|V_3^+|\leqslant 1$, $|V_3^-|\leqslant 1$, $(V_2^2\cup V_2^{12},V_3^{13})$ and $(V_2^1\cup V_2^{12},V_3^{23})$ are mixed pairs. Hence $q\leqslant 1+\binom{p_2+p_3}{\lfloor\frac{p_2+p_3}{2}\rfloor}+\binom{p_1+p_3}{\lfloor\frac{p_1+p_3}{2}\rfloor}+1\leqslant 2+2\binom{p-2}{\lfloor\frac{p-2}{2}\rfloor}$. 

Let $(|A|,|B|,|C|,|E|)=(1,1,2,2)$. Then $C\cup E=V_1$. If $A=C\setminus E$ and $B=C\cap E$, we may assume $A=\{x_1\}$, $B=\{x_2\}$, $C=\{x_1,x_2\}$ and $E=\{x_2,x_3\}$. We have $V_2=V_2^1\cup V_2^{12}\cup V_2^2\cup V_2^{23}$ and $V_3=V_3^+\cup V_3^{13}\cup V_3^3\cup V_3^-$. By Lemma~\ref{lem-diam2-ViA-VjB}, we have $V_1\rightarrow V_3^+\rightarrow V_2$, $V_2\rightarrow V_3^-\rightarrow V_1$, $V_3^{13}\rightarrow V_2^1$ and $V_2^{23}\rightarrow V_3^3$. So $|V_3^+|\leqslant 1$, $|V_3^-|\leqslant 1$, $(V_2^2\cup V_2^{12}\cup V_2^{23},V_3^{13})$ and $(V_2^1\cup V_2^{12}\cup V_2^2,V_3^3)$ are mixed pairs. Hence $q\leqslant 1+\binom{p_2+p_3+p_4}{\lfloor\frac{p_2+p_3+p_4}{2}\rfloor}+\binom{p_1+p_2+p_3}{\lfloor\frac{p_1+p_2+p_3}{2}\rfloor}+1\leqslant 2+2\binom{p-1}{\lfloor\frac{p-1}{2}\rfloor}$. 

If $A=C\setminus E$ and $B=E\setminus C$, we may assume $A=\{x_1\}$, $B=\{x_3\}$, $C=\{x_1,x_2\}$ and $E=\{x_2,x_3\}$. We have $V_2=V_2^1\cup V_2^{12}\cup V_2^{23}\cup V_2^3$ and $V_3=V_3^+\cup V_3^{13}\cup V_3^2\cup V_3^-$. By Lemma~\ref{lem-diam2-ViA-VjB}, we have $V_1\rightarrow V_3^+\rightarrow V_2$, $V_2\rightarrow V_3^-\rightarrow V_1$, $V_3^{13}\rightarrow V_2^1\cup V_2^3$ and $V_2^{12}\cup V_2^{23}\rightarrow V_3^2$. So $|V_3^+|\leqslant 1$, $|V_3^-|\leqslant 1$, $(V_2^{12}\cup V_2^{23},V_3^{13})$ and $(V_2^1\cup V_2^3,V_3^2)$ are mixed pairs. Hence $q\leqslant 1+\binom{p_2+p_3}{\lfloor\frac{p_2+p_3}{2}\rfloor}+\binom{p_1+p_4}{\lfloor\frac{p_1+p_4}{2}\rfloor}+1\leqslant 2+\binom{p_1+p_2+p_3+p_4}{\lfloor\frac{p_1+p_2+p_3+p_4}{2}\rfloor}=2+\binom{p}{\lfloor\frac{p}{2}\rfloor}$. 
\end{proof}

By the above lemmas in this subsection, we have the following result. 

\begin{theorem}\label{thm-H=4}
Let $4\leqslant p\leqslant q$, and let $D$ be a strong orientation of $\K(3,p,q)$ with diameter two. If $|\mathbb{H}|=4$, then 
\begin{center}
  $q\leqslant \max\Bigl\{1+2\binom{p-2}{\lfloor\frac{p-2}{2}\rfloor}+\binom{p-1}{\lfloor\frac{p-1}{2}\rfloor}, 2+2\binom{p-1}{\lfloor\frac{p-1}{2}\rfloor}\Bigr\}$.
\end{center}
\end{theorem}

\subsection{There are exactly five nonempty sets in $\mathbb{H}$}\label{subsec-H=5}

In this subsection, we consider the case that there are exactly three empty sets in $\mathbb{H}$. The case $V_2^-=V_2^A=V_2^B=\emptyset$ in $D$ is the same as $V_2^+=V_2^{\overline{A}}=V_2^{\overline{B}}=\emptyset$ in $\rev(D)$ where $\overline{A}=[3]\setminus A$ and $\overline{B}=[3]\setminus B$. Hence we only need to consider the following three cases.

\begin{lemma}\label{lem-H=5-V2+V2AV2-}
Let $5\leqslant p\leqslant q$, and let $D$ be a strong orientation of $\K(3,p,q)$ with diameter two. If there are exactly three empty sets $V_2^+=V_2^A=V_2^-=\emptyset$ in $\mathbb{H}$ for some nonempty and proper subset $A$ of $[3]$, then $q\leqslant 2+\binom{p-2}{\lfloor\frac{p-2}{2}\rfloor}$.
\end{lemma}

\begin{proof}
The case $V_2^+=V_2^A=V_2^-=\emptyset$ in $D$ is the same as $V_2^+=V_2^{\overline{A}}=V_2^-=\emptyset$ in $\rev(D)$ where $\overline{A}=[3]\setminus A$. Hence we may assume $|A|=1$ and let $A=\{x_1\}$. Then $V_2=V_2^2\cup V_2^3\cup V_2^{12}\cup V_2^{13}\cup V_2^{23}$ and $V_3=V_3^+\cup V_3^1\cup V_3^-$. By Lemma~\ref{lem-diam2-ViA-VjB}, we have $V_1\rightarrow V_3^+\rightarrow V_2$, $V_2\rightarrow V_3^-\rightarrow V_1$, $V_2^{12}\cup V_2^{13}\rightarrow V_3^1$. So $|V_3^+|\leqslant 1$, $|V_3^-|\leqslant 1$, $(V_2^2\cup V_2^3\cup V_2^{23},V_3^1)$ is a mixed pair. Hence $q\leqslant 1+\binom{p-2}{\lfloor\frac{p-2}{2}\rfloor}+1=2+\binom{p-2}{\lfloor\frac{p-2}{2}\rfloor}$. 
\end{proof}

\begin{lemma}\label{lem-H=5-V2+V2AV2B}
Let $5\leqslant p\leqslant q$, and let $D$ be a strong orientation of $\K(3,p,q)$ with diameter two. If there are exactly three empty sets $V_2^+=V_2^A=V_2^B=\emptyset$ in $\mathbb{H}$ for two nonempty and proper subsets $A$, $B$ of $[3]$, then $q\leqslant 1+2\binom{p-3}{\lfloor\frac{p-3}{2}\rfloor}$.
\end{lemma}

\begin{proof}
If $(|A|,|B|)=(2,2)$, then $A\cup B=V_1$, we may assume $A=\{x_1,x_2\}$ and $B=\{x_2,x_3\}$. We have $V_2=V_2^1\cup V_2^2\cup V_2^3\cup V_2^{13}\cup V_2^-$ and $V_3=V_3^+\cup V_3^{12}\cup V_3^{23}$. We know $x_2\rightarrow V_3$, since $\diam(V_2,x_2)\leqslant 2$, then we get $V_2\rightarrow x_2$, this means $V_2^2=\emptyset$, which is a contradiction. Hence $(|A|,|B|)=(1,1)$ or $(1,2)$. 

Let $(|A|,|B|)=(1,1)$. We may assume $A=\{x_1\}$ and $B=\{x_2\}$. Then $V_2=V_2^3\cup V_2^{12}\cup V_2^{13}\cup V_2^{23}\cup V_2^-$ and $V_3=V_3^+\cup V_3^1\cup V_3^2$. By Lemma~\ref{lem-diam2-ViA-VjB}, we have $V_1\rightarrow V_3^+\rightarrow V_2$, $V_3\rightarrow V_2^-\rightarrow V_1$, $V_2^{12}\cup V_2^{13}\rightarrow V_3^1$, $V_2^{12}\cup V_2^{23}\rightarrow V_3^2$. So $|V_3^+|\leqslant 1$, $(V_2^3\cup V_2^{23},V_3^1)$ and $(V_2^3\cup V_2^{13},V_3^2)$ are mixed pairs. Hence $q\leqslant 1+\binom{p-3}{\lfloor\frac{p-3}{2}\rfloor}+\binom{p-3}{\lfloor\frac{p-3}{2}\rfloor}=1+2\binom{p-3}{\lfloor\frac{p-3}{2}\rfloor}$.

Let $(|A|,|B|)=(1,2)$. If $A\subseteq B$, we may assume $A=\{x_1\}$ and $B=\{x_1,x_2\}$. We have $V_2=V_2^2\cup V_2^3\cup V_2^{13}\cup V_2^{23}\cup V_2^-$ and $V_3=V_3^+\cup V_3^1\cup V_3^{12}$. We know $x_1\rightarrow V_3$, since $\diam(V_2,x_1)\leqslant 2$, then we have $V_2\rightarrow x_1$, this means $V_2^{13}=\emptyset$, which is a contradiction. 

So we have $A\cap B=\emptyset$, and we may suppose $A=\{x_1\}$ and $B=\{x_2,x_3\}$. Then $V_2=V_2^2\cup V_2^3\cup V_2^{12}\cup V_2^{13}\cup V_2^-$ and $V_3=V_3^+\cup V_3^1\cup V_3^{23}$. By Lemma~\ref{lem-diam2-ViA-VjB}, we have $V_1\rightarrow V_3^+\rightarrow V_2$, $V_3\rightarrow V_2^-\rightarrow V_1$, $V_2^{12}\cup V_2^{13}\rightarrow V_3^1$ and $V_3^{23}\rightarrow V_2^2\cup V_2^3$. So $|V_3^+|\leqslant 1$, $(V_2^2\cup V_2^3,V_3^1)$ and $(V_2^{12}\cup V_2^{13},V_3^{23})$ are mixed pairs. Hence $q\leqslant 1+\binom{p-3}{\lfloor\frac{p-3}{2}\rfloor}+\binom{p-3}{\lfloor\frac{p-3}{2}\rfloor}=1+2\binom{p-3}{\lfloor\frac{p-3}{2}\rfloor}$.
\end{proof}

\begin{lemma}\label{lem-H=5-V2AV2BV2C}
Let $5\leqslant p\leqslant q$, and let $D$ be a strong orientation of $\K(3,p,q)$ with diameter two. If there are exactly three empty sets $V_2^A=V_2^B=V_2^C=\emptyset$ in $\mathbb{H}$ for three nonempty and proper subsets $A$, $B$, $C$ of $[3]$, then $q\leqslant \binom{p-3}{\lfloor\frac{p-3}{2}\rfloor}+\binom{p-2}{\lfloor\frac{p-2}{2}\rfloor}$.
\end{lemma}

\begin{proof}
The situation of $V_2^A=V_2^B=V_2^C=\emptyset$ in $D$ is the same as $V_2^{\overline{A}}=V_2^{\overline{B}}=V_2^{\overline{C}}=\emptyset$ in $\rev(D)$ where $\overline{A}=[3]\setminus A$, $\overline{B}=[3]\setminus B$ and $\overline{C}=[3]\setminus C$. We may suppose $|A|\leqslant |B|\leqslant |C|$, and so we only need to consider $(|A|,|B|,|C|)=(1,1,1)$ or $(1,1,2)$. 

Let $(|A|,|B|,|C|)=(1,1,1)$. We may assume $A=\{x_1\}$, $B=\{x_2\}$, $C=\{x_3\}$. Then $V_2=V_2^+\cup V_2^{12}\cup V_2^{13}\cup V_2^{23}\cup V_2^-$ and $V_3=V_3^1\cup V_3^2\cup V_3^3$. By Lemma~\ref{lem-diam2-ViA-VjB}, we have $V_1\rightarrow V_2^+\rightarrow V_3$, $V_3\rightarrow V_2^-\rightarrow V_1$, $V_2^{12}\rightarrow V_3^1\cup V_3^2$, $V_2^{13}\rightarrow V_3^1\cup V_3^3$ and $V_2^{23}\rightarrow V_3^2\cup V_3^3$. So $(V_2^{23},V_3^1)$, $(V_2^{13},V_3^2)$ and $(V_2^{12},V_3^3)$ are mixed pairs. Let $p_1=|V_2^{12}|$, $p_2=|V_2^{13}|$, $p_3=|V_2^{23}|$. Then $1\leqslant p_1,p_2,p_3\leqslant p-4$ and $p_1+p_2+p_3=p-2$. Hence $q\leqslant 1+\binom{p_3}{\lfloor\frac{p_3}{2}\rfloor}+\binom{p_2}{\lfloor\frac{p_2}{2}\rfloor}+\binom{p_1}{\lfloor\frac{p_1}{2}\rfloor}\leqslant \binom{p_1+p_2+p_3}{\lfloor\frac{p_1+p_2+p_3}{2}\rfloor}=\binom{p-2}{\lfloor\frac{p-2}{2}\rfloor}$. 

Let $(|A|,|B|,|C|)=(1,1,2)$. If $A\cup B=C$, we may assume $A=\{x_1\}$, $B=\{x_2\}$ and $C=\{x_1,x_2\}$, then $V_2=V_2^+\cup V_2^3\cup V_2^{13}\cup V_2^{23}\cup V_2^-$ and $V_3=V_3^1\cup V_3^2\cup V_3^{12}$. We know $V_3\rightarrow x_3$, since $\diam(x_3,V_2)\leqslant 2$, then we get $x_3\rightarrow V_2$, this means $V_2^-=\emptyset$, which is a contradiction. So we may suppose $A\subseteq C$ and $B\cup C=V_1$, and let $A=\{x_1\}$, $B=\{x_3\}$ and $C=\{x_1,x_2\}$. Then $V_2=V_2^+\cup V_2^2\cup V_2^{13}\cup V_2^{23}\cup V_2^-$ and $V_3=V_3^1\cup V_3^{12}\cup V_3^3$. By Lemma~\ref{lem-diam2-ViA-VjB}, we have $V_1\rightarrow V_2^+\rightarrow V_3$, $V_3\rightarrow V_2^-\rightarrow V_1$, $V_3^{12}\rightarrow V_2^2$, $V_2^{13}\rightarrow V_3^1\cup V_3^3$ and $V_2^{23}\rightarrow V_3^3$. So $(V_2^2\cup V_2^{23},V_3^1)$, $(V_2^{13}\cup V_2^{23},V_3^{12})$ and $(V_2^2,V_3^{3})$ are mixed pairs. Let $p_1=|V_2^2|$, $p_2=|V_2^{13}|$, $p_3=|V_2^{23}|$. Then $1\leqslant p_1,p_2,p_3\leqslant p-4$ and $p_1+p_2+p_3=p-2$. Hence $q\leqslant \binom{p_1+p_3}{\lfloor\frac{p_1+p_3}{2}\rfloor}+\binom{p_2+p_3}{\lfloor\frac{p_2+p_3}{2}\rfloor}+\binom{p_1}{\lfloor\frac{p_1}{2}\rfloor}\leqslant \binom{p_1+p_3}{\lfloor\frac{p_1+p_3}{2}\rfloor}+\binom{p_1+p_2+p_3}{\lfloor\frac{p_1+p_2+p_3}{2}\rfloor}\leqslant \binom{p-3}{\lfloor\frac{p-3}{2}\rfloor}+\binom{p-2}{\lfloor\frac{p-2}{2}\rfloor}$. 
\end{proof}

By the above lemmas in this subsection, we have the following result. 

\begin{theorem}\label{thm-H=5}
Let $5\leqslant p\leqslant q$, and let $D$ be a strong orientation of $\K(3,p,q)$ with diameter two. If $|\mathbb{H}|=5$, then 
\begin{center}
  $q\leqslant \binom{p-3}{\lfloor\frac{p-3}{2}\rfloor}+\binom{p-2}{\lfloor\frac{p-2}{2}\rfloor}$.
\end{center}
\end{theorem}

\subsection{There are exactly six nonempty sets in $\mathbb{H}$}\label{subsec-H=6}

Suppose there are exactly six nonempty sets in $\mathbb{H}$, then we have $p\geqslant 6$. Equivalently, there are exactly two empty sets in $\mathbb{H}$. The case $V_2^-=V_2^A=\emptyset$ in $D$ is the same as $V_2^+=V_2^{\overline{A}}=\emptyset$ in $\rev(D)$ where $\overline{A}=[3]\setminus A$. 

Let $6\leqslant p\leqslant q$, and let $D$ be a strong orientation of $\K(3,p,q)$. If there are exactly two empty sets $V_2^+=V_2^-=\emptyset$ in $\mathbb{H}$, then $V_3=V_3^+\cup V_3^-$, and so $q=|V_3|\leqslant 2$, which is a contradiction. If there are exactly two empty sets $V_2^+=V_2^A=\emptyset$ in $\mathbb{H}$ for some nonempty and proper subset $A$ of $[3]$, then $V_3=V_3^+\cup V_3^A$. Let $B=[3]\setminus A$, $V_{11}=\{x_i\mid i\in A\}$ and $V_{12}=\{x_j\mid j\in B\}$. Then $V_1\rightarrow V_3^+\rightarrow V_2$ and $V_{11}\rightarrow V_3^A\rightarrow V_{12}$. We have $V_{11}\rightarrow V_3$. Since $\diam(V_2,V_{11})\leqslant 2$, then we have $V_2\rightarrow V_{11}$, this means there are at most four nonempty sets in $\mathbb{H}$, which is a contradiction. Hence we only need to consider the following case.

\begin{lemma}\label{lem-H=6-V2AV2B}
Let $6\leqslant p\leqslant q$, and let $D$ be a strong orientation of $\K(3,p,q)$ with diameter two. If there are exactly two empty sets $V_2^A=V_2^B=\emptyset$ in $\mathbb{H}$ for some nonempty and proper subsets $A$ and $B$ of $[3]$, then $q\leqslant \binom{p-2}{\lfloor\frac{p-2}{2}\rfloor}$.
\end{lemma}

\begin{proof}
The case $V_2^A=V_2^B=\emptyset$ in $D$ is the same as $V_2^{\overline{A}}=V_2^{\overline{B}}=\emptyset$ in $\rev(D)$ where $\overline{A}=[3]\setminus A$ and $\overline{B}=[3]\setminus B$. We may suppose $|A|\leqslant |B|$, and so we only need to consider $(|A|,|B|)=(1,1)$ or $(1,2)$. 

If $(|A|,|B|)=(1,1)$, we may assume $A=\{x_1\}$, $B=\{x_2\}$. Then $V_2=V_2^+\cup V_2^{12}\cup V_2^{13}\cup V_2^{23}\cup V_2^3\cup V_2^-$ and $V_3=V_3^1\cup V_3^2$. We know $V_3\rightarrow x_3$, since $\diam(x_3,V_2)\leqslant 2$, then we have $x_3\rightarrow V_2$. This means $V_2^{12}=V_2^-=\emptyset$, which is a contradiction. 

Let $(|A|,|B|)=(1,2)$. If $A\subseteq B$, we may assume $A=\{x_1\}$, $B=\{x_1,x_2\}$. Then $V_2=V_2^+\cup V_2^{13}\cup V_2^2\cup V_2^{23}\cup V_2^3\cup V_2^-$ and $V_3=V_3^1\cup V_3^{12}$. We know $V_3\rightarrow x_3$, since $\diam(x_3,V_2)\leqslant 2$, then we have $x_3\rightarrow V_2$. This means $V_2^2=V_2^-=\emptyset$, which is a contradiction. 

Suppose $A\cap B=\emptyset$, we may assume $A=\{x_1\}$ and $B=\{x_2,x_3\}$. Then $V_2=V_2^+\cup V_2^{12}\cup V_2^{13}\cup V_2^2\cup V_2^3\cup V_2^-$ and $V_3=V_3^1\cup V_3^{23}$. By Lemma~\ref{lem-diam2-ViA-VjB}, we have $V_1\rightarrow V_2^+\rightarrow V_3$, $V_3\rightarrow V_2^-\rightarrow V_1$, $V_3^{23}\rightarrow V_2^2\cup V_2^3$, $V_2^{12}\cup V_2^{13}\rightarrow V_3^1$. So $(V_2^2\cup V_2^3,V_3^1)$, $(V_2^{12}\cup V_2^{13},V_3^{23})$ are mixed pairs. Let $p_1=|V_2^{12}\cup V_2^{13}|$, $p_2=|V_2^2\cup V_2^3|$. Then $2\leqslant p_1,p_2\leqslant p-4$ and $p_1+p_2=p-2$. Hence $q\leqslant \binom{p_2}{\lfloor\frac{p_2}{2}\rfloor}+\binom{p_1}{\lfloor\frac{p_1}{2}\rfloor}\leqslant \binom{p_1+p_2}{\lfloor\frac{p_1+p_2}{2}\rfloor}=\binom{p-2}{\lfloor\frac{p-2}{2}\rfloor}$. 
\end{proof}

By the above discussion and lemma in this subsection, we have the following result. 

\begin{theorem}\label{thm-H=6}
Let $6\leqslant p\leqslant q$, and let $D$ be a strong orientation of $\K(3,p,q)$ with diameter two. If $|\mathbb{H}|=6$, then 
\begin{center}
  $q\leqslant \binom{p-2}{\lfloor\frac{p-2}{2}\rfloor}$.
\end{center}
\end{theorem}

\subsection{There are exactly seven nonempty sets in $\mathbb{H}$}\label{subsec-H=7}

Suppose there are exactly seven nonempty sets in $\mathbb{H}$, then we have $p\geqslant 7$. Equivalently, there is exactly one empty set in $\mathbb{H}$. Let $7\leqslant p\leqslant q$, and let $D$ be a strong orientation of $\K(3,p,q)$. 

If there is exactly one empty set $V_2^+=\emptyset$ (or, $V_2^-=\emptyset$) in $\mathbb{H}$, then $V_3=V_3^+$ (or, $V_3=V_3^-$), and so $q=|V_3|\leqslant 1$, which is a contradiction. 

If there is exactly one empty set $V_2^A=\emptyset$ in $\mathbb{H}$ for some nonempty and proper subset $A$ of $[3]$, then $V_3=V_3^A$. Let $B=[3]\setminus A$, $V_{11}=\{x_i\mid i\in A\}$ and $V_{12}=\{x_j\mid j\in B\}$. Then $V_{11}\rightarrow V_3^A\rightarrow V_{12}$, i.e., $V_{11}\rightarrow V_3\rightarrow V_{12}$. Since $\diam(V_2,V_{11})\leqslant 2$ and $\diam(V_{12},V_2)\leqslant 2$, then we have $V_{12}\rightarrow V_2\rightarrow V_{11}$, this means $V_2=V_2^B$, i.e., there is exactly one nonempty set in $\mathbb{H}$, which is a contradiction.

By the above discussion in this subsection, we have the following result. 

\begin{theorem}\label{thm-H=7}
Let $7\leqslant p\leqslant q$, and let $D$ be a strong orientation of $\K(3,p,q)$. Then $|\mathbb{H}|\neq 7$.
\end{theorem}

\subsection{The main result of the big part}\label{subsec-big-part}

When $p\geqslant 5$, we know
\begin{eqnarray*}
  1+\binom{p-3}{\lfloor\frac{p-3}{2}\rfloor}+\binom{p-1}{\lfloor\frac{p-1}{2}\rfloor}+\binom{p}{\lfloor\frac{p}{2}\rfloor} & \leqslant & \binom{p+1}{\lfloor\frac{p+1}{2}\rfloor}-1, \\
  \max\Bigl\{1+2\binom{p-2}{\lfloor\frac{p-2}{2}\rfloor}+\binom{p-1}{\lfloor\frac{p-1}{2}\rfloor}, 2+2\binom{p-1}{\lfloor\frac{p-1}{2}\rfloor}\Bigr\} & < & \binom{p+1}{\lfloor\frac{p+1}{2}\rfloor}-1, \\
  \binom{p-3}{\lfloor\frac{p-3}{2}\rfloor}+\binom{p-2}{\lfloor\frac{p-2}{2}\rfloor} & < & \binom{p+1}{\lfloor\frac{p+1}{2}\rfloor}-1.
\end{eqnarray*}
By Theorem~\ref{thm-H=1}, Theorem~\ref{thm-H=2}, Theorem~\ref{thm-H=3}, Theorem~\ref{thm-H=4}, Theorem~\ref{thm-H=5}, Theorem~\ref{thm-H=6}, Theorem~\ref{thm-H=7}, we get the main result of the big part.

\begin{theorem}\label{thm-big-part}
Let $5\leqslant p\leqslant q$, and suppose the complete tripartite graph $\K(3,p,q)$ has a strong orientation with diameter two. Then 
\begin{equation*}
  q\leqslant \binom{p+1}{\lfloor\frac{p+1}{2}\rfloor}-1.
\end{equation*}
\end{theorem}


\section{Small part: find strong orientations with diameter two}\label{sec-small-part}

Let $5\leqslant p\leqslant q\leqslant \binom{p+1}{\lfloor\frac{p+1}{2}\rfloor}-1$, in this small part, we give strong orientations of complete tripartite graph $\K(3,p,q)$ with diameter two.

We first give a $(p,q,\lambda)$-orientation of complete bipartite graph $\K(p,q)$ with diameter three.

\begin{orientation}[$(p,q,\lambda)$-orientation]\label{orientation-p-q-orientation}
Let $4\leqslant p\leqslant q\leqslant \binom{p}{\lambda}$ where $\lambda\geqslant 2$ and $p-\lambda\geqslant 2$, and let the vertex-set of $\K(p,q)$ be $W=W_1\cup W_2$ where the two biparts are 
\begin{center}
  $W_1=\{x_0,x_1,\ldots,x_{p-1}\}$ and $W_2=\{y_0,y_1,\ldots,y_{q-1}\}$. 
\end{center} 
Let $D$ be an orientation of $\K(p,q)$ such that 
\begin{center}
  $W_2(\lambda)\subseteq W_2$,
\end{center} 
and all the neighbour sets 
\begin{center}
  $N_D^+(y)\cap W_1\in \binom{W_1}{\lambda}$ are different for all $y\in W_2$ 
\end{center}
(and so these neighbour sets are independent with each other), where 
\begin{center}
  $W_2(\lambda)=\{y_i^*\mid 0\leqslant i<p\}$
\end{center} 
satisfying for $0\leqslant i<p$ that 
\begin{center}
  $N_D^+(y_i^*)\cap W_1=\{x_j\in W_1\mid j\equiv i+s\pmod p, 0\leqslant s\leqslant \lambda-1\}$,\\ 
  $N_D^-(y_i^*)\cap W_1=\{x_j\in W_1\mid j\equiv i+s\pmod p, \lambda\leqslant s\leqslant p-1\}$.
\end{center} 
Then $\diam(W_1,W_1)\leqslant 2$, $\diam(W_2,W_2)\leqslant 2$, and so $\diam(D)=3$.
\end{orientation}

\begin{proof}
For any two vertices $y,w\in W_2$, we know the neighbour sets $N_D^+(y)\cap W_1$ and $N_D^+(w)\cap W_1$ are independent, i.e., 
\begin{center}
  $(N_D^+(y)\cap W_1)\setminus (N_D^+(w)\cap W_1)\neq\emptyset$,\\ $(N_D^+(w)\cap W_1)\setminus (N_D^+(y)\cap W_1)\neq\emptyset$,  
\end{center} 
and so $\partial_D(y,w)\leqslant 2$ and $\partial_D(w,y)\leqslant 2$. This means $\diam(W_2,W_2)\leqslant 2$.

For $x_i\in W_1$, we have 
\begin{center}
  $N_D^+(x_i)\cap W_2(\lambda)=\{y_j^*\in W_2(\lambda)\mid j\equiv i+s\pmod p, 1\leqslant s\leqslant p-\lambda\}$,~~~~~~\\ 
  $N_D^-(x_i)\cap W_2(\lambda)=\{y_j^*\in W_2(\lambda)\mid j\equiv i+s\pmod p, p-\lambda+1\leqslant s\leqslant p\}$.
\end{center} 
In the following, all the subscripts are read module $p$. Take any $x_i,x_j\in W_1$, we may assume $0\leqslant i<j\leqslant p-1$. If $1\leqslant j-i\leqslant \lambda-1$, then 
\begin{center}
  $0\leqslant j-i-1\leqslant \lambda-2$ and $1\leqslant i-j+\lambda\leqslant \lambda-1$, 
\end{center} 
and so 
\begin{center}
  $x_i\rightarrow y_{i+1}^*\rightarrow x_{(i+1)+(j-i-1)}=x_j$ and \\
  $x_j\rightarrow y_{j+p-\lambda}^*\rightarrow x_{(j+p-\lambda)+(i-j+\lambda)}=x_{i+p}=x_i$.
\end{center} 
If $\lambda\leqslant j-i\leqslant p-1$, then 
\begin{center}
  $1\leqslant j-i-\lambda+1\leqslant p-\lambda$ and $1\leqslant i-j+p\leqslant p-\lambda$, 
\end{center} 
and so 
\begin{center}
  $x_i\rightarrow y_{i+(j-i-\lambda+1)}^*=y_{j-\lambda+1}^*\rightarrow x_{(j-\lambda+1)+(\lambda-1)}=x_j$ and \\
  $x_j\rightarrow y_{j+(i-j+p)}^*=y_{i+p}^*=y_i^*\rightarrow x_i$.
\end{center} 
Hence we have the distances $\partial_D(x_i,x_j)\leqslant 2$ and $\partial_D(x_j,x_i)\leqslant 2$. This means $\diam(W_1,W_1)\leqslant 2$. The diameter of $D$ is easily got by the above results. 
\end{proof}

By the discussions and main result Theorem~\ref{thm-big-part} of big part, we know the bound 
\begin{center}
  $q\leqslant \binom{p+1}{\lfloor\frac{p+1}{2}\rfloor}-1$
\end{center} 
occurs in Theorem~\ref{thm-H=2}, and more precisely in Lemma~\ref{lem-H=2-V2AV2B-neq}. 

Suppose $5\leqslant p\leqslant q$, and let $D$ be a strong orientation of $\K(3,p,q)$ with diameter two as in the proof of Lemma~\ref{lem-H=2-V2AV2B-neq}. In this orientation, we let $V_2=V_2^A\cup V_2^B$ satisfying $|A|=1$, $|B|=2$, $A\subseteq B$, and $A$ and $B$ are two subsets of $[3]$, we may assume $A=\{x_1\}$ and $B=\{x_1,x_2\}$. By the proof of Lemma~\ref{lem-H=2-V2AV2B-neq}, we have 
\begin{center}
  $V_2=V_2^1\cup V_2^{12}$, $x_1\rightarrow V_2\rightarrow x_3$, $x_3\rightarrow V_3\rightarrow x_1$, $V_3=V_3^{23}\cup V_3^3$. 
\end{center} 
The necessary and sufficient conditions of $\diam(D)=2$ are 
\begin{center}
  $\diam(V_2,V_2)\leqslant 2$ and $\diam(V_3,V_3)\leqslant 2$; 
\end{center}
\begin{center}
  $N_D^+(y)\cap V_3^3\neq\emptyset$ for any $y\in V_2^{12}$ (implied by $\diam(V_2^{12},x_2)\leqslant 2$),\\
  $N_D^-(y)\cap V_3^{23}\neq\emptyset$ for any $y\in V_2^1$ (implied by $\diam(x_2,V_2^1)\leqslant 2$),\\
  $N_D^+(z)\cap V_2^1\neq\emptyset$ for any $z\in V_3^{23}$ (implied by $\diam(V_3^{23},x_2)\leqslant 2$),\\
  $N_D^-(z)\cap V_2^{12}\neq\emptyset$ for any $z\in V_3^3$ (implied by $\diam(x_2,V_3^3)\leqslant 2$); 
\end{center} 
and 
\begin{center}
  $N_D^+(y)\cap V_3\neq\emptyset$ and $N_D^-(y)\cap V_3\neq\emptyset$ for any $y\in V_2$\\
  (implied by $\diam(V_2,x_1)\leqslant 2$ and $\diam(x_3,V_2)\leqslant 2$),\\
  $N_D^+(z)\cap V_2\neq\emptyset$ and $N_D^-(z)\cap V_2\neq\emptyset$ for any $z\in V_3$\\
  (implied by $\diam(V_3,x_3)\leqslant 2$ and $\diam(x_1,V_3)\leqslant 2$).
\end{center} 

In the following, we can use these conditions and $(p,q,\lambda)$-orientations to give some orientations of $\K(3,p,q)$ with diameter two.

\begin{orientation}[Odd orientation]\label{orientation-odd-2p+}
Let $k\geqslant 2$ and $5\leqslant p=2k+1\leqslant q$. Then $\lfloor\frac{p}{2}\rfloor=k$ and $\lceil\frac{p}{2}\rceil=\lfloor\frac{p+1}{2}\rfloor=k+1$. Let $D$ be a strong orientation of $\K(3,p,q)$ with $V_2=V_2^1\cup V_2^{12}$ and $V_3=V_3^{23}\cup V_3^3$. Let $p_1=|V_2^1|=k+1$, $p_2=|V_2^{12}|=k$, $q_1=|V_3^{23}|$, $q_2=|V_3^3|$, 
\begin{center}
  $F_1=D[V_2\cup V_3^{23}]$ and $F_2=D[V_2\cup V_3^3]$.
\end{center} 
Then $F_1$ is an orientation of $\K(p,q_1)$ and $F_2$ is an orientation of $\K(p,q_2)$. As described in Orientation~\ref{orientation-p-q-orientation}, suppose $F_1$ is a $(p,q_1,k+1)$-orientation of $\K(p,q_1)$ and $F_2$ is a $(p,q_2,k)$-orientation of $\K(p,q_2)$ such that 
\begin{center}
  $V_3^{23}(k+1)\subseteq V_3^{23}$, $\{N_D^+(z)\cap V_2\mid z\in V_3^{23}\}\subseteq \binom{V_2}{k+1}$ and\\   
  $V_3^3(k)\subseteq V_3^3$, $\{N_D^+(z)\cap V_2\mid z\in V_3^3\}\subseteq \binom{V_2}{k}\setminus\{V_2^{12}\}$.~~~~~  
\end{center}
Then $\diam(D)=2$ and $2p\leqslant q\leqslant \binom{p+1}{\lfloor\frac{p+1}{2}\rfloor}-1$.
\end{orientation}

\begin{proof}
By the property of $(p,q,\lambda)$-orientations in Orientation~\ref{orientation-p-q-orientation}, we know $\diam(V_2,V_2)\leqslant 2$, $\diam(V_3^{23},V_3^{23})\leqslant 2$ and $\diam(V_3^3,V_3^3)\leqslant 2$. Since $V_3^3\rightarrow x_2\rightarrow V_3^{23}$, we have $\diam(V_3^3,V_3^{23})\leqslant 2$. Since $V_3^{23}$ is associated with $(p,q_1,k+1)$-orientation and $V_3^3$ is associated with $(p,q_2,k)$-orientation, then $|N_D^+(z)\cap V_2)|=k+1$, $|N_D^+(w)\cap V_2|=k$ for any $z\in V_3^{23}$ and any $w\in V_3^3$, and so $(N_D^+(z)\cap V_2)\setminus (N_D^+(w)\cap V_2)\neq\emptyset$, which implies $\partial_D(z,w)\leqslant 2$, and so $\diam(V_3^{23},V_3^3)\leqslant 2$. Hence we have $\diam(V_3,V_3)\leqslant 2$. 

By the property of $(p,q,\lambda)$-orientations, for any $y\in V_2$, we know 
\begin{center}
  $|N_D^+(y)\cap V_3^{23}|\geqslant |N_D^+(y)\cap V_3^{23}(k+1)|=k$, ~~~~~\\ 
  $|N_D^-(y)\cap V_3^{23}|\geqslant |N_D^-(y)\cap V_3^{23}(k+1)|=k+1$, \\ 
  $|N_D^+(y)\cap V_3^3|\geqslant |N_D^+(y)\cap V_3^3(k)|=k+1$, ~~~~~~~\\ 
  $|N_D^-(y)\cap V_3^3|\geqslant |N_D^-(y)\cap V_3^3(k)|=k$;~~~~~~~~~~~~~
\end{center} 
for any $z\in V_3^{23}$, we have 
\begin{center}
  $|N_D^+(z)\cap V_2|=k+1$, 
  $|N_D^-(z)\cap V_2|=k$;
\end{center}
for any $z\in V_3^3$, we have 
\begin{center}
  $|N_D^+(z)\cap V_2|=k$, 
  $|N_D^-(z)\cap V_2|=k+1$. 
\end{center} 
Hence we have $N_D^+(y)\cap V_3\neq\emptyset$ and $N_D^-(y)\cap V_3\neq\emptyset$ for any $y\in V_2$, $N_D^+(z)\cap V_2\neq\emptyset$ and $N_D^-(z)\cap V_2\neq\emptyset$ for any $z\in V_3$; and $N_D^+(y)\cap V_3^3\neq\emptyset$ for any $y\in V_2^{12}$, $N_D^-(y)\cap V_3^{23}\neq\emptyset$ for any $y\in V_2^1$. For any $z\in V_3^3$, we know $N_D^+(z)\cap V_2\neq V_2^{12}$, which means $N_D^-(z)\cap V_2^{12}\neq\emptyset$. For any $z\in V_3^{23}$, we have $|N_D^+(z)\cap V_2|=k+1>|V_2^{12}|$, which means $N_D^+(z)\cap V_2^1\neq\emptyset$. 

By the above discussion, we thus have $\diam(D)=2$, and we also get 
\begin{center}
  $p\leqslant q_1\leqslant \binom{p}{k+1}$ and $p\leqslant q_2\leqslant \binom{p}{k}-1$, 
\end{center} 
hence $2p\leqslant q=q_1+q_2\leqslant \binom{p}{k+1}+\binom{p}{k}-1=\binom{p+1}{k+1}-1$.
\end{proof}

\begin{orientation}[Even orientation]\label{orientation-even-2p+}
Let $k\geqslant 2$ and $6\leqslant p=2k+2\leqslant q$. Then $\lfloor\frac{p}{2}\rfloor=\lceil\frac{p}{2}\rceil=\lfloor\frac{p+1}{2}\rfloor=k+1$ and $\lceil\frac{p+1}{2}\rceil=k+2$. Let $D$ be a strong orientation of $\K(3,p,q)$ with $V_2=V_2^1\cup V_2^{12}$ and $V_3=V_3^{23}\cup V_3^3$. Let $p_1=|V_2^1|=k$, $p_2=|V_2^{12}|=k+2$, $q_1=|V_3^{23}|$, $q_2=|V_3^3|$, 
\begin{center}
  $F_1=D[V_2\cup V_3^{23}]$ and $F_2=D[V_2\cup V_3^3]$.
\end{center} 
Then $F_1$ is an orientation of $\K(p,q_1)$ and $F_2$ is an orientation of $\K(p,q_2)$. As described in Orientation~\ref{orientation-p-q-orientation}, suppose $F_1$ is a $(p,q_1,k+2)$-orientation of $\K(p,q_1)$ and $F_2$ is a $(p,q_2,k+1)$-orientation of $\K(p,q_2)$ such that 
\begin{center}
  $V_3^{23}(k+2)\subseteq V_3^{23}$, $\{N_D^+(z)\cap V_2\mid z\in V_3^{23}\}\subseteq \binom{V_2}{k+2}\setminus\{V_2^{12}\}$ and\\   
  $V_3^3(k+1)\subseteq V_3^3$, $\{N_D^+(z)\cap V_2\mid z\in V_3^3\}\subseteq \binom{V_2}{k+1}$.~~~~~~~~~~~~~~~~~~  
\end{center}
Then $\diam(D)=2$ and $2p\leqslant q\leqslant \binom{p+1}{\lfloor\frac{p+1}{2}\rfloor}-1$.
\end{orientation}

\begin{proof}
By the property of $(p,q,\lambda)$-orientations in Orientation~\ref{orientation-p-q-orientation}, we know $\diam(V_2,V_2)\leqslant 2$, $\diam(V_3^{23},V_3^{23})\leqslant 2$ and $\diam(V_3^3,V_3^3)\leqslant 2$. Since $V_3^3\rightarrow x_2\rightarrow V_3^{23}$, we have $\diam(V_3^3,V_3^{23})\leqslant 2$. Since $V_3^{23}$ is associated with $(p,q_1,k+2)$-orientation and $V_3^3$ is associated with $(p,q_2,k+1)$-orientation, then $|N_D^+(z)\cap V_2)|=k+2$, $|N_D^+(w)\cap V_2|=k+1$ for any $z\in V_3^{23}$ and any $w\in V_3^3$, and so $(N_D^+(z)\cap V_2)\setminus (N_D^+(w)\cap V_2)\neq\emptyset$, which implies $\partial_D(z,w)\leqslant 2$, and so $\diam(V_3^{23},V_3^3)\leqslant 2$. Hence we have $\diam(V_3,V_3)\leqslant 2$. 

By the property of $(p,q,\lambda)$-orientations, for any $y\in V_2$, we know 
\begin{center}
  $|N_D^+(y)\cap V_3^{23}|\geqslant |N_D^+(y)\cap V_3^{23}(k+2)|=k$, ~~~~~\\ 
  $|N_D^-(y)\cap V_3^{23}|\geqslant |N_D^-(y)\cap V_3^{23}(k+2)|=k+2$, \\ 
  $|N_D^+(y)\cap V_3^3|\geqslant |N_D^+(y)\cap V_3^3(k+1)|=k+1$, ~~\\ 
  $|N_D^-(y)\cap V_3^3|\geqslant |N_D^-(y)\cap V_3^3(k+1)|=k+1$;~~~
\end{center} 
for any $z\in V_3^{23}$, we have 
\begin{center}
  $|N_D^+(z)\cap V_2|=k+2$, 
  $|N_D^-(z)\cap V_2|=k$;
\end{center}
for any $z\in V_3^3$, we have 
\begin{center}
  $|N_D^+(z)\cap V_2|=k+1$, 
  $|N_D^-(z)\cap V_2|=k+1$. 
\end{center} 
Hence we have $N_D^+(y)\cap V_3\neq\emptyset$ and $N_D^-(y)\cap V_3\neq\emptyset$ for any $y\in V_2$, $N_D^+(z)\cap V_2\neq\emptyset$ and $N_D^-(z)\cap V_2\neq\emptyset$ for any $z\in V_3$; and $N_D^+(y)\cap V_3^3\neq\emptyset$ for any $y\in V_2^{12}$, $N_D^-(y)\cap V_3^{23}\neq\emptyset$ for any $y\in V_2^1$. For any $z\in V_3^3$, we have $|N_D^+(z)\cap V_2|=k+1<|V_2^{12}|$, which means $N_D^-(z)\cap V_2^{12}\neq\emptyset$. For any $z\in V_3^{23}$, we know $N_D^+(z)\cap V_2\neq V_2^{12}$, which means $N_D^+(z)\cap V_2^1\neq\emptyset$. 

By the above discussion, we thus have $\diam(D)=2$, and we also get 
\begin{center}
  $p\leqslant q_1\leqslant \binom{p}{k+2}-1$ and $p\leqslant q_2\leqslant \binom{p}{k+1}$, 
\end{center} 
hence $2p\leqslant q=q_1+q_2\leqslant \binom{p}{k+2}-1+\binom{p}{k+1}=\binom{p+1}{k+2}-1$.
\end{proof}

By Orientation~\ref{orientation-odd-2p+} and Orientation~\ref{orientation-even-2p+}, we get that $\K(3,p,q)$ has an orientation of diameter two for any $p\geqslant 5$ and any $q$ with $2p\leqslant q\leqslant \binom{p+1}{\lfloor\frac{p+1}{2}\rfloor}-1$. 

\begin{theorem}\label{thm-small-part-2p+}
Let $p\geqslant 5$ and $2p\leqslant q\leqslant \binom{p+1}{\lfloor\frac{p+1}{2}\rfloor}-1$, then $f(\K(3,p,q))=2$. 
\end{theorem}

In the rest of this section, we need to find diameter two orientations of $\K(3,p,q)$ for $5\leqslant p\leqslant q\leqslant 2p$.

\begin{orientation}\label{orientation-2p-}
Let $5\leqslant p\leqslant q$. Let $D$ be a strong orientation of $\K(3,p,q)$ with $V_2=V_2^1\cup V_2^2\cup V_2^3$ and $V_3=V_3^{23}\cup V_3^{13}\cup V_3^{12}$ where $V_2^1=\{y_1\}$, $V_2^2=\{y_2\}$, $p_1=|V_2^3|=p-2$, $V_3^{23}=\{z_{23}\}$, $V_3^{13}=\{z_{13}\}$ and $q_1=|V_3^{12}|$. We use Lemma~\ref{lem-diam2-ViA-VjB} to give partial orientations of $D$: $V_3^{12}\rightarrow V_2^1\cup V_2^2$, $V_3^{13}\rightarrow V_2^1\cup V_2^3$ and $V_3^{23}\rightarrow V_2^2\cup V_2^3$. Additionally, we let $V_2^1\rightarrow V_3^{23}$, $V_2^2\rightarrow V_3^{13}$ and $F=D[V_2^3\cup V_3^{12}]$ be a $(p_1,q_1,\lambda)$-orientation of $\K(p_1,q_1)$ where $\lambda=\lfloor\frac{p_1}{2}\rfloor$ or $\lceil\frac{p_1}{2}\rceil$. Then 
\begin{center}
  $V_2\cup V_3\subseteq N_D^+(V_1)$, \\ 
  $V_2\cup V_3\subseteq N_D^-(V_1)$, \\ 
  $V_1\cup V_3\subseteq N_D^+(V_2)$, \\ 
  $V_1\cup V_3\subseteq N_D^-(V_2)$, \\ 
  $V_1\cup V_2\subseteq N_D^+(V_3)$, \\  
  $V_1\cup V_2\subseteq N_D^-(V_3)$,
\end{center} 
$\diam(D)=2$ and $p\leqslant q\leqslant 2+\binom{p-2}{\lfloor\frac{p-2}{2}\rfloor}$.
\end{orientation}

\begin{proof}
The results can be checked by calculating distances between all pairs of vertices. 
\end{proof}

Let $D$ be the orientation as in Orientation~\ref{orientation-2p-}, then we have 
\begin{center}
  $V_2\cup V_3=N_D^+(V_1)=N_D^-(V_1)$, \\ 
  $V_1\cup V_3=N_D^+(V_2)=N_D^-(V_2)$, \\ 
  $V_1\cup V_2=N_D^+(V_3)=N_D^-(V_3)$.
\end{center} 
If we add $V_3^+$ (or $V_3^-$, or both) to $V_3$ and denote $V_3'=V_3\cup V_3^+$ (or $V_3'=V_3\cup V_3^-$, or $V_3'=V_3\cup V_3^+\cup V_3^-$) and the new orientation as $D'$, then we have $V_1\rightarrow V_3^+\rightarrow V_2\rightarrow V_3^-\rightarrow V_1$, and we can check that $\diam(V_3^+,V(D))\leqslant 2$ and $\diam(V(D),V_3^+)\leqslant 2$ (or, $\diam(V_3^-,V(D))\leqslant 2$ and $\diam(V(D),V_3^-)\leqslant 2$). Hence $\diam(D')=2$ and $p+1\leqslant q\leqslant 3+\binom{p-2}{\lfloor\frac{p-2}{2}\rfloor}$ (or $p+2\leqslant q\leqslant 4+\binom{p-2}{\lfloor\frac{p-2}{2}\rfloor}$). By the above discussion and Orientation~\ref{orientation-2p-}, we get that $\K(3,p,q)$ has an orientation of diameter two for any $p\geqslant 5$ and any $q$ with $p\leqslant q\leqslant 4+\binom{p-2}{\lfloor\frac{p-2}{2}\rfloor}$.

\begin{theorem}\label{thm-small-part-p+}
Let $p\geqslant 5$ and $p\leqslant q\leqslant 4+\binom{p-2}{\lfloor\frac{p-2}{2}\rfloor}$, then $f(\K(3,p,q))=2$.
\end{theorem}

\begin{orientation}\label{orientation-p+3}
Let $5\leqslant p\leqslant q$. Let $D$ be a strong orientation of $\K(3,p,q)$ with $V_2=V_2^1\cup V_2^2\cup V_2^{12}$ and $V_3=V_3^+\cup V_3^{13}\cup V_3^{23}\cup V_3^{3}$ where $V_2^1=\{y_1\}$, $V_2^2=\{y_2\}$, $V_2^{12}=\{y_{12}\}\cup V_2'$, $V_3^+=\{z_+\}$, $V_3^{13}=\{z_{13}\}$, $V_3^{23}=\{z_{23}\}$ and $q_1=|V_3^3|$. We use Lemma~\ref{lem-diam2-ViA-VjB} to give partial orientations of $D$: $V_3^+\rightarrow V_2$, $V_3^{13}\rightarrow V_2^1$ and $V_3^{23}\rightarrow V_2^2$. Additionally, we let $V_3^{13}\rightarrow V_2^2\cup V_2'$, $V_3^{23}\rightarrow V_2^1\cup V_2'$, $y_{12}\rightarrow V_3^{13}\cup V_3^{23}$ and $F=D[V_2\cup V_3^3]$ be a $(p,q_1,\lambda)$-orientation of $\K(p,q_1)$ where $\lambda=\lfloor\frac{p}{2}\rfloor$ or $\lceil\frac{p}{2}\rceil$. Then $\diam(D)=2$ and $p+3\leqslant q\leqslant 3+\binom{p}{\lfloor\frac{p}{2}\rfloor}$.
\end{orientation}

\begin{proof}
The results can be checked by calculating distances between all pairs of vertices. 
\end{proof}

By Orientation~\ref{orientation-2p-}, we get that $\K(3,p,q)$ has an orientation of diameter two for any $p\geqslant 5$ and any $q$ with $p+3\leqslant q\leqslant 3+\binom{p}{\lfloor\frac{p}{2}\rfloor}$.

\begin{theorem}\label{thm-small-part-p+3}
Let $p\geqslant 5$ and $p+3\leqslant q\leqslant 3+\binom{p}{\lfloor\frac{p}{2}\rfloor}$, then $f(\K(3,p,q))=2$.
\end{theorem}

We know for $p\geqslant 5$ that 
\begin{center}
  $p+2\leqslant 4+\binom{p-2}{\lfloor\frac{p-2}{2}\rfloor}$,\\
  $\max\Bigl\{2p,4+\binom{p-2}{\lfloor\frac{p-2}{2}\rfloor},3+\binom{p}{\lfloor\frac{p}{2}\rfloor}\Bigr\}=3+\binom{p}{\lfloor\frac{p}{2}\rfloor}$.
\end{center} 
So we have the following result. 

\begin{theorem}\label{thm-small-part-p-to-2p}
Let $p\geqslant 5$ and $p\leqslant q\leqslant 2p$, then $f(\K(3,p,q))=2$.
\end{theorem}

Combing Theorem~\ref{thm-small-part-2p+} and Theorem~\ref{thm-small-part-p-to-2p}, we get the small part.

\begin{theorem}\label{thm-small-part}
Let $5\leqslant p\leqslant q\leqslant \binom{p+1}{\lfloor\frac{p+1}{2}\rfloor}-1$, then $f(\K(3,p,q))=2$.
\end{theorem}



\section{Acknowledgement}


This work is supported by the National Natural Science Foundation of China (11901014) and Yale-NUS Seed Grant (IG23-SG009).


%
%

\end{document}